\documentclass[a4paper,12pt]{article}
\usepackage{amsmath,amsthm,amsfonts,amssymb,bbm}
\usepackage{graphicx,psfrag,tikz}
\usepackage{subcaption}
\usepackage{cite}
\usepackage{hyperref}
\usepackage[hmargin=1in,vmargin=1in]{geometry}

\renewcommand{\P}{\mathbb{P}}
\newcommand{\E}{\mathbb E}

\newcommand{\R}{\mathbb R}
\newcommand{\Z}{\mathbb Z}
\DeclareMathOperator{\Var}{Var}

\newcommand{\e}{\varepsilon}

\providecommand{\abs}[1]{\vert#1\vert}
\newcommand{\cA}{\mathcal{A}}
\newcommand{\cB}{\mathcal{B}}
\newcommand{\cC}{\mathcal{C}}
\newcommand{\cD}{\mathcal{D}}
\newcommand{\cE}{\mathcal{E}}
\newcommand{\cF}{\mathcal{F }}
\newcommand{\cG}{\mathcal{G}}
\newcommand{\cH}{\mathcal{H}}
\newcommand{\cI}{\mathcal{I}}
\newcommand{\cL}{\mathcal{L}}
\newcommand{\cO}{\mathcal{O}}
\newcommand{\cR}{\mathcal{R}}
\newcommand{\cU}{\mathcal{U}}
\newcommand{\cV}{\mathcal{V}}
\newcommand{\cY}{\mathcal{Y}}
\newcommand{\cZ}{\mathcal{Z}}

\newcommand{\argmax}[1]{\underset{#1}{\mathrm{argmax}}}

\newtheorem{theorem}{Theorem}[section]
\newtheorem{lemma}[theorem]{Lemma}
\newtheorem{proposition}[theorem]{Proposition}
\newtheorem{corollary}[theorem]{Corollary}
\newtheorem{assumption}[theorem]{Assumption}
\newtheorem{definition}[theorem]{Definition}
\theoremstyle{remark}
\newtheorem{remark}[theorem]{Remark}

\def\arr{a}

\def\serv{s}
\def\depa{d}

\def\arrv{\mathbf\arr}

\def\servv{\mathbf\serv}
\def\depav{\mathbf\depa}

\numberwithin{equation}{section}
\numberwithin{figure}{section}

\author{Ofer Busani\thanks{University of Bristol, School of Mathematics, Fry Building, Woodland Rd., Bristol BS8 1UG, UK. E-mail: {\tt o.busani@bristol.ac.uk}} \and Patrik L.\ Ferrari\thanks{Institute for Applied Mathematics, Bonn University, Endenicher Allee 60, 53115 Bonn, Germany. E-mail: {\tt ferrari@uni-bonn.de}}}

\title{Universality of the geodesic tree in last passage percolation}

\date{\today}

\begin{document}

\sloppy

\maketitle
\begin{abstract}
In this paper we consider the geodesic tree in exponential last passage percolation. We show that for a large class of initial conditions around the origin, the line-to-point geodesic that terminates in a cylinder of width $o(N^{2/3})$ and length $o(N)$ agrees in the cylinder, with the stationary geodesic sharing the same end point. In the case of the point-to-point model, we consider width $\delta N^{2/3}$ and length up to $\delta^{3/2} N/(\log(\delta^{-1}))^3$ and provide lower and upper bound for the probability that the geodesics agree in that cylinder.
\end{abstract}

\section{Introduction}
The last passage percolation model (LPP) is one of the most well studied models in the Kardar-Parisi-Zhang (KPZ) universality class of stochastic growth models. In this model, to each site $(i,j)\in\Z^2$, one associates an independent random variable $\omega_{i,j}$ exponentially distributed with parameter one. In the simplest case, the point-to-point LPP model, for a given point $(m,n)$ in the first quadrant, one defines the last passage time as
\begin{equation}\label{eq1}
 G(m,n)=\max_{\pi:(0,0)\to (m,n)} \sum_{(i,j)\in\pi} \omega_{i,j},
\end{equation}
where the maximum is taken over all up-right paths, that is, paths whose incremental steps are either $(1,0)$ or $(0,1)$. Consider the spatial direction $x$ to be $(1,-1)$ and the time direction $t$ to be $(1,1)$. Then one defines a height function $h(x,t=N)=G(N+x,N-x)$, see~\cite{Jo03b} and also~\cite{PS00,PS02} for a continuous analogue related to the Hammersley process~\cite{Ham72}. The height function has been studied extensively: at time $N$, it has fluctuations of order $N^{1/3}$ and non-trivial correlation over distance $N^{2/3}$, which are the KPZ scaling exponents~\cite{KPZ86,KMH92,BKS85}. Furthermore, both the one-point distributions~\cite{BR99b,BR99,PS00,BDJ99,Jo00b,CFS16} as well as the limiting processes~\cite{Jo03b,PS02,BFPS06,BFS07b,BFP09} are known for a few initial conditions (or geometries in LPP framework). Finally, the correlations in time of the interface are non-trivial over macroscopic distances~\cite{Fer08,CFP10b} and they have been recently partially studied~\cite{Jo18,BG18,JR19,BL17,FO18}.

Another very interesting but less studied aspect of models in the KPZ universality class is the geometrical properties of the geodesics (also known as maximizers). For the LPP model, geodesics are the paths achieving the maximum in \eqref{eq1}. In the case of the exponential random LPP described above, for any end-point $(m,n)$, there is a unique geodesic. Geodesics follow characteristic directions and if the end-point is at distance $\cO(N)$ from the origin, then it has spatial fluctuations of order $\cO(N^{2/3})$ with respect to the line joining the origin with the end-point~\cite{Jo00,BSS14}.

Consider two or more end-points. To each of the end-points there is one geodesic from $(0,0)$ and thus the set of geodesics as seen from the end-points in the direction of the origin, have a non-trivial coalescing structure. Some recent studies of this structure in LPP and related models can be found in~\cite{BBS20,BGHH20,Ham16,Pim16,BSS17,Ham17}. One might expect that on a large scale the coalescing structure is universal and thus not depending on the details of the chosen random variables defining the LPP models (provided, of course, that the model is still in the KPZ class, which rules out, for instance, heavy tailed random variable).

The fact that the height function decorrelates over macroscopic times is reflected in the geometrical behaviour of the geodesics. For instance, taking two end-points at distance $\cO(N^{2/3})$ of each other, the coalescence point of the two geodesics will be at distance of order $\cO(N)$ from the end-points~\cite{FS03b} and have a non-trivial distribution over the full macroscopic scale, as already noticed in some numerical studies in~\cite{FerPhD}. More refined recent results are also available~\cite{SS19,Zha19}.

In the study of the covariance of the time-time correlations~\cite{FO18} it was proven that taking one end-point as $(N,N)$ and the second $(\tau N,\tau N)$, then as $\tau\to 1$, the first order correction to the covariance of the LPP is $\cO((1-\tau)^{2/3}N^{2/3})$ and is completely independent of the geometry of the LPP, i.e., it is the same whether one considers the point-to-point LPP as in \eqref{eq1} or the line-to-point LPP, for which the geodesics start from a point on the antidiagonal crossing the origin. This suggests that the coalescing structure of the end-points in $\{(N+k,N-k),|k|\leq \delta N^{2/3}\}$, for a small $\delta>0$, should be independent of the LPP geometry over a time-span $o(N)$ from the end-points. In particular, the coalescing structure should be locally the same as the one from the stationary model, introduced in~\cite{BCS06}. In~\cite{BBS20} a result in this direction has been proven. Among other results, they showed that the tree of point-to-point geodesics starting from every vertex in a box of side length $\delta N^{2/3}$ going to a point at distance $N$ agree inside the box with the tree of  stationary geodesics.

The goal of this work is to improve on previous results in the following points:
\begin{enumerate}
	\item In the case of point-to-point LPP, we extend previous results by showing (Theorem~\ref{thm:coal}) that the coalescence to the stationary geodesics holds with high probability for any geodesic starting in a large box around the origin and terminating in a cylinder whose width is of order $N^{2/3}$ and its length is of order $N$ (see Figure~\ref{Fig:boxes}). In other words, we obtain the correct dimensions of the cylinder around the point $(N,N)$.
	\item In the case of point-to-point LPP, we improve the lower bound of the coalescence result from exponent $3/8$ to the correct exponent $1/2$ (Theorem~\ref{thm:coal}). In the process of proving it, we provide a simple probabilistic proof (Theorem~\ref{thm:locCylinder}) for the concentration of geodesics around their characteristics with the optimal exponential decay.
	\item In the case of point-to-point LPP, we obtain an upper bound on the coalescence event (Theorem~\ref{thm:LBcoal}) that differs from the lower bound only by a logarithmic factor, i.e., we have indeed obtained the correct exponent.
	\item In the case of general initial conditions we obtain (Theorem~\ref{thm:coalGeneralIC}) a lower bound on the probability that the geodesic tree agree with that of stationary one in a cylinder of width $N^{2/3}$ and length $N$. The order of the lower bound depends on the concentration of the exit point of the geodesics around the origin.
\end{enumerate}

Another problem that is closely related to the coalescence of the point-to-point geodesic with the stationary one is the question of coalescence of point-to-point geodesics. More precisely, consider the probability that two infinite geodesics starting $k^{2/3}$ away from each other will coalesce after $R k$ steps. A lower bound of the order $C R^{-c}$ was obtained in~\cite{Pim16}. Matching upper bound together with the identification of the constant $c=-2/3$ was found in~\cite{BSS17}. The analogue result for the point-to-point coalescence was completed  more recently in~\cite{Zha19}. Finally, in~\cite{BBS20} it is proven that the infinite geodesics in fact coalesce with their point-to-point counterparts and identified the polynomial decay obtained in \cite{Zha19}.

In a second type of coalescence results, one considers the probability that geodesics leaving from two points that are located at distance of order $N^{2/3}$ away from each other and terminate at or around $(N,N)$ coalesce. In the setup of Brownian LPP, one takes $k$ geodesics leaving from a small interval of order $\epsilon N^{2/3}$ and terminating at time $N$ in an interval of the same order. Then, the probability that they are disjoint is of order $\epsilon^{(k^2-1)/2}$ with a subpolynomial correction, see~\cite[Theorem 1.1]{Ham17}. It was conjectured there that the lower bound should have the same exponent. For $k=2$ this is proven in \cite[Theorem 2.4]{BGH19}. Furthermore, an upper bound of order $\tau^{2/9}$ on the probability that two geodesics starting from the points $(0,0)$ and $(0,N^{2/3})$ and terminating at $(N,N)$ do not coalesce by time $(1-\tau)N$ is obtained in~\cite[Theorem 2.8]{BBS20}.

Our Theorem~\ref{thm:coal2} gives the exact exponent $1/2$, for the probability that any two geodesics starting from a large box of dimensions of order $N\times N^{2/3}$ and terminating at a common point in a small box of size $N\times N^{2/3}$ coalesce (see Figure~\ref{Fig:boxes} for  more accurate dimensions). Theorem~\ref{thm:coal2} can be compared with rarity of disjoint geodesics that was considered in \cite{Ham17} although for geodesics starting from a big box rather than a small one.

What is then the reason for the discrepancy in the different exponents (exponent $3/2$ in the results in \cite{Ham17} and the $1/2$ exponent in this paper)? Clearly, the geometry is different as in this paper we consider geodesics starting from a large box around the origin, as opposed to a small box of size $\delta N^{2/3}$ in \cite{Ham17}. Let us try to give a heuristic argument for a possible settlement of this discrepancy. Let us divide the interval $I:=\{(0,i)\}_{0 \leq i \leq N^{2/3}}$ into $\delta^{-1}$ sub-intervals of size $\delta N^{2/3}$. If the event that two geodesics starting from $I$ do not meet by the time they reach the small interval around the point $(N,N)$
 is dominated by the event that the two geodesics leave from the same small sub-interval and if these events decorrelate on the scale of $N^{2/3}$ then by \cite[Theorem 1.1]{Ham17} we have roughly $\delta^{-1}$ decorrelated events of probability (up to logarithmic correction) $\delta^{3/2}$ which would imply that the probability of two geodesics starting from $I$ to not meet by the time they reach a small interval around $(N,N)$ is (up to logarithmic correction) $\delta^{-1}\delta^{3/2}=\delta^{1/2}$.

Concerning the methods used in this paper, one input we use is a control over the lateral fluctuations of the geodesics in the LPP. In Theorem~\ref{thm:locCylinder} we show that the probability that the geodesic of the point-to-point LPP is not localized around a distance $M N^{2/3}$ from the characteristic line decay like $e^{-c M^3}$, which is the optimal power of the decay. This is proven using the approach of~\cite{BSS14}, see Theorem~\ref{prop1}, once the mid-point analogue estimate is derived, see Theorem~\ref{thm:coal}. The novelty here is a simple and short proof of this latter by using only comparison with stationary models. This probabilistic method is much simpler than previous ones.

To prove Theorem~\ref{thm:coal}, the first step is to prove that with high probability the spatial trajectories of both the geodesic of the point-to-point LPP as well as the one of the stationary model with density $1/2$ are sandwiched between the geodesics for the stationary models with some densities $\rho_+>1/2$ and $\rho_-<1/2$ respectively. This then reduces the problem to finding bounds on the coalescing probability only for the two geodesics of the stationary models. This is done using the coupling between different stationary models introduced in~\cite{FS18}. The main ingredient to prove Theorem~\ref{thm:LBcoal} is to show that the geodesics of the stationary models with different densities did not coalesce too early with some positive probability. Here, the application of the queueing representation of the coupling in~\cite{FS18} is more delicate than the one needed for the lower bound, as we now have to force the geodesic away from each other.

\paragraph{Outline of the paper.} In Section~\ref{SectResults} we define the model and state the main results. We recall in Section~\ref{SectPreliminaries} some recurrent notations and basic results on stationary LPP. In Section~\ref{SectLocalStat} we prove first Theorem~\ref{thm:locCylinder} on the localization of the point-to-point geodesics and then show that the geodesics can be sandwiched between two version of the stationary model, see Lemma~\ref{cor:og}. This allows us to prove Theorems~\ref{thm:coal} and~\ref{thm:coalGeneralIC} in Section~\ref{sectLowerBound}. Section~\ref{SectUpperBound} deals with the proof of Theorem~\ref{thm:LBcoal} and Theorem \ref{thm:coal2}.

\paragraph{Acknowledgments.} The authors are grateful to M\'arton Bal\'azs for initial discussions on the topic that led to our collaboration.
O.\ Busani was supported by the EPSRC EP/R021449/1 Standard Grant of the UK. This study did not involve any underlying data. The work of P.L. Ferrari was partly funded by the Deutsche Forschungsgemeinschaft (DFG, German Research Foundation) under Germany’s Excellence Strategy - GZ 2047/1, projekt-id 390685813 and by the Deutsche Forschungsgemeinschaft (DFG, German Research Foundation) - Projektnummer 211504053 - SFB 1060.

\section{Main results}\label{SectResults}
	Let $\omega=\{\omega_x\}_{x\in \Z^2}$ be i.i.d.\ $\textrm{Exp}(1)$-distributed random weights on the vertices of $\Z^2$. For $o\in \Z^2$, define the last-passage percolation (LPP) process on $o+\Z_{\geq0}^2$ by
	\begin{equation}\label{v:G}
	G_{o,y}=\max_{x_{\bullet}\,\in\,\Pi_{o,y}}\sum_{k=0}^{\abs{y-o}_1}\omega_{x_k}\quad\text{ for } y\in o+\Z_{\geq 0}^2.
	\end{equation}
	$\Pi_{o,y}$ is the set of paths $x_{\bullet}=(x_k)_{k=0}^n$ that start at $x_0=o$, end at $x_n=y$ with $n=\abs{y-o}_1$, and have increments $x_{k+1}-x_k\in\{\mathrm{e}_1,\mathrm{e}_2\}$. The a.s.\ unique path $\pi_{o,y}\in \Pi_{o,y}$ that attains the maximum in \eqref{v:G} is the {\it geodesic} from $o$ to $y$.

 Let $\cL=\{x\in\Z^2| x_1+x_2=0\}$ be the antidiagonal crossing through the origin. Given some random variables (in general non independent) $\{h_0(x)\}_{x\in \cL}$ on $\cL$, independent from $\omega$, define the last passage time with initial condition $h_0$ by
	\begin{equation}\label{eq1.2}
G^{h_0}_{\cL,y}=\max_{x_{\bullet}\,\in\,\Pi_{\cL,y}}\bigg(h_0(x_0)+\sum_{k=1}^{\abs{y-o}_1}\omega_{x_k}\bigg)\quad \textrm{for } y>\cL,
	\end{equation}
where $y>\cL$ is meant in the sense of the order on the lattice. We also denote by $Z^{h_0}_{{\cL,y}}$ the point $x_0$ from where the geodesic from $\cL$ to $y$ leaves the line $\cL$, and refer to it as the \emph{exit point} of the last passage percolation with initial condition $h_0$.

One can define stationary models parameterized by a density $\rho\in (0,1)$, both for the LPP on the positive quadrant as for the LPP on the north-east of $\cL$, see Section~\ref{SectStatLPP} for detailed explanations. In that case we denote the stationary LPP by $G^\rho_{o,y}$ or $G^\rho_{\cL,y}$ respectively.

	For $\sigma\in\R_+$ and $0<\tau<1$ we define the cylinder of width $\sigma N^{2/3}$ and length $\tau N$
	\begin{equation}
	\cC^{\sigma,\tau}=\{i\mathrm{e}_4+j\mathrm{e}_3:(1-\tau) N\leq i \leq N,-\tfrac\sigma2 N^{2/3}\leq j \leq \tfrac\sigma2 N^{2/3}\}.
	\end{equation}
	Similarly, for $\sigma\in\R_+$ and $0<\tau<1$ we define a set of width $\sigma N^{2/3}$ and length $\tau N$
	\begin{equation}\label{R}
	\cR^{\sigma,\tau}=\{i\mathrm{e}_4+j\mathrm{e}_3:0 \leq i \leq \tau N,-\tfrac\sigma2 N^{2/3}\leq j \leq \tfrac\sigma2 N^{2/3},|j|<i\}.
	\end{equation}
	\begin{remark}
		Note that the shape of $\cR^{\sigma,\tau}$ in \eqref{R} is somewhat different than the blue cylinder in Figure~\ref{Fig:boxes}. The reason for that is that in this paper we use exit points with respect to the vertical and horizontal axis so that the shape defined in \eqref{R} is easier to work with. We stress that similar results can be obtained for a box as in Figure~\ref{Fig:boxes} by using exit points with respect to the antidiagonal $\cL$.
	\end{remark}
 Due to the correspondence to stochastic growth models in the KPZ universality class, we denote the time direction by $(1,1)$ and the spatial direction by $(1,-1)$. In particular, for any $0<\tau<1$ define the time horizon
	\begin{equation}
	L_\tau=\{\tau N \mathrm{e}_4+i\mathrm{e}_3:-\infty<i<\infty\},
	\end{equation}
	Let $x,y,z\in\Z^2$ be such that $x,y\leq z$. For the geodesics $\pi_{x,z}$ and $\pi_{y,z}$ we define the coalescence point
	\begin{equation}
	C_p(\pi_{x,z},\pi_{y,z})=\inf\{u\in\Z^2:u\in \pi_{x,z}\cap\pi_{y,z}\},
	\end{equation}
where the infimum is with respect to the order $\leq$ on the lattice.
\subsubsection*{Upper and lower bounds on the coalescing point}
The first main result of this paper is that, with probability going to $1$ as $\delta\to 0$, the set of geodesics ending at any point in the cylinder $\cC^{\delta,\tau}$ of the stationary LPP with density $1/2$ is indistinguishable from the geodesics of the point-to-point LPP from the origin for any $\tau\leq \delta^{3/2}/(\log(\delta^{-1}))^3$.
\begin{theorem}\label{thm:coal} Let $o=(0,0)$.
There exist $C,\delta_0>0$ such that for any $\delta\in (0,\delta_0)$ and $\tau\leq\delta^{3/2}/(\log(\delta^{-1}))^3$,
\begin{equation}
\P\Big(C_p(\pi^{1/2}_{o,x},\pi_{y,x})\leq L_{1-\tau} \quad \forall x\in \cC^{\delta,\tau},y\in \cR^{\frac18\log\delta^{-1},1/4}\Big)\geq 1- C\delta^{1/2}\log(\delta^{-1})
\end{equation}
for all $N$ large enough.
\end{theorem}

As a direct corollary we have that in a cylinder of spatial width $o(N^{2/3})$ and time width $o(N)$ around the end-point $(N,N)$, the geodesics are indistinguishable from the stationary ones.
\begin{corollary}\label{cor:coal}
For any $\e>0$,
\begin{equation}
\lim_{N\to\infty}
\P\Big(C_p(\pi^{1/2}_{o,x},\pi_{o,x})\leq L_{1-N^{-\e}} \quad \forall x\in \cC^{N^{-\e},N^{-\e}}\Big) =1.
\end{equation}
\end{corollary}

Theorem~\ref{thm:coal} generalizes to LPP with a large class of initial conditions, i.e., consider LPP from $\cL$ with initial condition $h_0$, like the ones considered in~\cite{CFS16,FO18}.
We make the following assumption on $h_0$. Recall the exit point $Z^{h_0}_{\cL,x}$ mentioned right after \eqref{eq1.2}.
\begin{assumption}\label{ass_IC}
Let $x^1=N\mathrm{e}_4 + \tfrac34 \delta N^{2/3} \mathrm{e}_3$ and $x^2=N\mathrm{e}_4 - \tfrac34 \delta N^{2/3} \mathrm{e}_3$. Assume that
\begin{equation}
\P(Z^{h_0}_{\cL,x^1}\leq \log(\delta^{-1}) N^{2/3})\geq 1-Q(\delta)
\end{equation}
and
\begin{equation}
\P(Z^{h_0}_{\cL,x^2}\geq -\log(\delta^{-1}) N^{2/3})\geq 1-Q(\delta)
\end{equation}
for all $N$ large enough, with a function $Q(\delta)$ satisfying $\lim_{\delta\to 0}Q(\delta)=0$.
\end{assumption}
Under Assumption~\ref{ass_IC} the analogue of Theorem~\ref{thm:coal} (and thus of Corollary~\ref{cor:coal}) holds true.
\begin{theorem}\label{thm:coalGeneralIC}
Under Assumption~\ref{ass_IC}, there exist $C,\delta_0>0$ such that for any $\delta\in (0,\delta_0)$ and $\tau\leq\delta^{3/2}/(\log(\delta^{-1}))^3$,
\begin{equation}
\P\Big(C_p(\pi^{1/2}_{\cL,x},\pi^{h_0}_{\cL,x})\leq L_{1-\tau} \quad \forall x\in \cC^{\delta,\tau}\Big)\geq 1- C\delta^{1/2}\log(\delta^{-1})- Q(\delta)
\end{equation}
for all $N$ large enough.
\end{theorem}

The exponent $1/2$ in Theorem~\ref{thm:coal} is optimal as our next result shows.
\begin{theorem}\label{thm:LBcoal} Let $o=(0,0)$.
	There exist $C,\delta_0>0$ such that for any $\delta\in (0,\delta_0)$ and $\tau\leq\delta^{3/2}/(\log(\delta^{-1}))^3$,
	\begin{equation}
	\P\Big(C_p(\pi^{1/2}_{o,x},\pi_{y,x})\leq L_{1-\tau} \quad \forall x\in \cC^{\delta,\tau},y\in \cR^{(\frac18\log\delta^{-1}),\tau}\Big)\leq 1- C\delta^{1/2}
	\end{equation}
	for all $N$ large enough.
\end{theorem}
The following result is closely related to Theorem~\ref{thm:coal} and Theorem~\ref{thm:LBcoal}, it considers the question of coalescence of point-to-point geodesics.
\begin{theorem}\label{thm:coal2}
	There exist $C,\delta_0>0$ such that for any $\delta\in (0,\delta_0)$ and $\tau\leq\delta^{3/2}/(\log(\delta^{-1}))^3$,
	\begin{equation}
	1-C\delta^{1/2}\log(\delta^{-1})\leq \P\Big(C_p(\pi_{w,x},\pi_{y,x})\leq L_{1-\tau} \quad \forall x\in \cC^{\delta,\tau},w,y\in \cR^{(\frac18\log\delta^{-1}),\tau}\Big)\leq 1- C\delta^{1/2}
	\end{equation}
	for all $N$ large enough.
\end{theorem}

\subsubsection*{Cubic decay of localization}
In order to prove the main theorems we will need some control on the spatial fluctuations of the geodesics for the point-to-point problem. As this estimate has its own interest, we state it below as Theorem~\ref{thm:loc}. For a an up-right path $\gamma$ we denote
\begin{equation}
\begin{aligned}
\Gamma^u_k(\gamma)&=\max\{l:(k,l)\in\gamma\},\\
\Gamma^l_k(\gamma)&=\min\{l:(k,l)\in\gamma\}.
\end{aligned}
\end{equation}
When $\gamma$ is a geodesic associated with a direction, we denote the direction by $\xi=(\xi_1,\xi_2)$ with $\xi_1+\xi_2=1$, and we set
\begin{equation}
\Gamma_k(\gamma)=\max\{|\Gamma^u_k(\gamma)-\tfrac{\xi_2}{\xi_1} k|,|\Gamma^l_k(\gamma)-\tfrac{\xi_2}{\xi_1} k|\}.
\end{equation}

\begin{theorem}\label{thm:locCylinder}
Let $\e\in (0,1]$. Then there exists $N_0(\e)$ and $c_1(\e)$ such that for $\xi$ satisfying $\e\leq \xi_2/\xi_1\leq 1/\e$,
	\begin{equation}
	\P\big(\Gamma_{k}(\pi_{o,\xi N})>M(\tau N)^{2/3}\textrm{ for all }k\in[0,\tau \xi_1 N] \big)\leq e^{-c_1M^3}
	\end{equation}
for all $\tau N\geq N_0$ and all $M\leq (\tau N)^{1/3}/\log(N)$.
\end{theorem}
A statement similar to Theorem~\ref{thm:locCylinder} with Gaussian bound is Proposition~2.1 of~\cite{BG18}. The authors employed Theorems~10.1 and~10.5 of~\cite{BSS14}. Potentially their argument could be improved to get a cubic decay using the bounds from random matrices of~\cite{LR10}, but we did not verify this. Instead, we provide a short and self-contained proof of the localization result using comparison with stationarity only.

\begin{remark}\label{remptptlocal}
Theorem~\ref{thm:locCylinder} states the optimal localization scale for small $\tau$. By symmetry of the point-to-point problem, the same statement holds with $\tau$ replaced by $1-\tau$ and and gives the optimal localization scale for $\tau$ close to $1$.
\end{remark}

\subsubsection*{A family of random initial conditions.}
Now let us consider the family of initial conditions, interpolating between flat initial condition (i.e., point-to-line LPP) and the stationary initial condition, for which the time-time covariance was studied in~\cite{FO18}. For $\sigma\geq 0$, let us define
\begin{equation}\label{eq2.10}
h_0(k,-k)=\sigma \times \left\{
\begin{array}{ll}
\sum_{\ell=1}^k(X_\ell-Y_\ell),&\textrm{ for }k\geq 1,\\
0,&\textrm{ for }k=0,\\
-\sum_{\ell=k+1}^0 (X_\ell-Y_\ell),&\textrm{ for }k\leq -1.
\end{array}\right.
\end{equation}
where $\{X_k,Y_k\}_{k\in\Z}$ are independent random variables $X_k,Y_k\sim\textrm{Exp}(1/2)$. For $\sigma=0$ it corresponds to the point-to-line LPP, while for $\sigma=1$ it is the stationary case with density $1/2$.
\begin{proposition}
For LPP with initial condition \eqref{eq2.10}, Assumption~\ref{ass_IC} holds with
\begin{equation}
 Q(\delta)= C e^{-c (\log(\delta^{-1}))^3}
\end{equation}
for some constants $C,c>0$.
\end{proposition}
\begin{proof}
The estimates leading to the proofs are all contained in the proof of Lemma~5.2 of~\cite{FO18}, see part \emph{(b) Random initial conditions}. Replacing in that proof $\tau=1$, $M=\tilde M = 2^{-2/3} \tfrac34 \delta$ and $\alpha=2^{-2/3} \log(\delta^{-1})$ we obtain $Q(\delta)\leq \min\{C e^{-c (\alpha-M)^3},C e^{-c (\alpha-M)^4/\alpha}\}$. As $\alpha-M\simeq \alpha$ for small $\delta$, the result follows.
\end{proof}
\begin{figure}
\begin{center}
	\begin{tikzpicture}[scale=0.65, every node/.style={transform shape}]
	node/.style={transform shape}]
	\def\s{1.2}
	\draw [color=red,fill=red!20] (10,10) -- (10.5,9.5) -- (7.5,6.5)--(6.5,7.5) --(9.5,10.5) --(10,10) ;
	\draw [color=blue,fill=blue!20] (0,0) -- (-1.5,1.5) -- (2,5)--(5,2) --(1.5,-1.5) -- (0,0) ;
	\draw[red,line width=2pt] plot [smooth,tension=0.8] coordinates{(-1,1) (0,2) (2,3.6) (3,5) (4,5.5) (5,6) (6.2,6.7) (8,8) (9,9.2)};
	\draw[red,line width=2pt] plot [smooth,tension=0.8] coordinates{(-1,1) (0,2) (2,3.6) (3,5) (4,5.5) (5,6) (6.2,6.5) (8,7.5) (9,8.5)} ;
	\draw[red,line width=2pt] plot [smooth,tension=0.8] coordinates{(-1,1) (0,2) (2,3.6) (3,5) (4,5.5) (5,6) (6.2,6.7) (8,8.5) (8.7,9.6)};
	\draw[red,line width=2pt] plot [smooth,tension=0.8] coordinates{ (8,8) (9,8.7) (9.5,9.2)};
	
	\draw[line width=0.5pt] plot [smooth,tension=0.8] coordinates{(2,2) (3,4) (4,5) (5,6) (6.2,6.7) (8,8) (9,9.2)};
	\draw[line width=0.5pt] plot [smooth,tension=0.8] coordinates{(2,2) (3,4) (4,5) (5,6) (6.2,6.5) (8,7.5) (9,8.5)} ;
	\draw[line width=0.5pt] plot [smooth,tension=0.8] coordinates{(2,2) (3,4) (4,5) (5,6) (6.2,6.7) (8,8.5) (8.7,9.6)};
	\draw[line width=0.5pt] plot [smooth,tension=0.8] coordinates{ (8,8) (9,8.7) (9.5,9.2)};
	\draw [<->](9.8,10.8)--(10.8,9.8);
	\node [scale=\s,rotate=-45] at (10.6,10.6) {$\delta N^{2/3}$};
	\draw [<->](6.3,7.7)--(9.3,10.7);
	\node [scale=\s,rotate=45] at (7.5,9.5) {$\delta^{3/2}\log(\delta^{-1}) N$};
	\draw[<->] (2.2,5.2) -- (5.2,2.2);
	\node [scale=\s,rotate=-45] at (4.7,3.5) {$\log(\delta^{-1}) N^{2/3}$};
	\draw[<->] (-1.8,1.8) -- (1.7,5.3);
	\node [scale=\s,rotate=45] at (-0.7,3.7) {$N/4$};
	\draw[<->] (1.7,-1.7) -- (11.7,8.3);
	\node [scale=\s,rotate=45] at (6.5,2.5) {$N$};
	\draw [<->](0,10)--(0,0) -- (10,0);
 	\end{tikzpicture}
	\caption{Illustration of Theorems~\ref{thm:coal} and~\ref{thm:LBcoal}. With high probability, any geodesic tree (black curve) consisting of all geodesics starting from a fixed point in the blue cylinder and terminating at any point in the red cylinder, will agree in the red box with the stationary tree (red curve) of intensity $1/2$.}\label{Fig:boxes}
\end{center}
\end{figure}
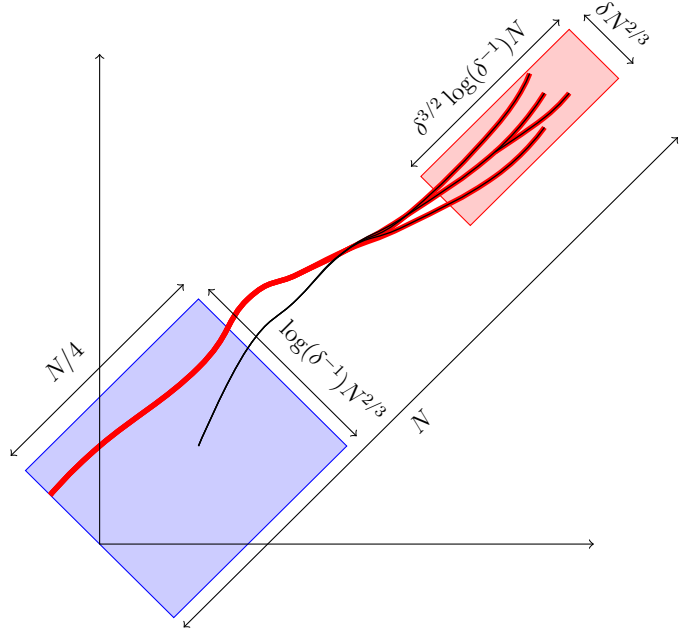

\section{Preliminaries}\label{SectPreliminaries}

\subsection{Some general notation}
We mention here some of the notations which will be used throughout the paper. We denote
$\Z_{\ge0}=\{0,1,2,3, \dotsc\}$ and $\Z_{>0}=\{1,2,3,\dotsc\}$. We use four standard vectors in $\R^2$, namely $\mathrm{e}_1=(1,0)$ and $\mathrm{e}_2=(0,1)$. We also denote $\mathrm{e}_3=(1,-1)$ and $\mathrm{e}_4=(1,1)$. The direction $\mathrm{e}_3$ represents the \emph{spatial direction}, while $\mathrm{e}_4$ represents the \emph{temporal direction}. Furthermore, for a point $x=(x_1,x_2)\in\R^2$ the $\ell^1$-norm is $\abs{x}=\abs{x_1} + \abs{x_2}$. We also use the partial ordering on $\R^2$: for $x=(x_1,x_2)\in\R^2$ and $y=(y_1,y_2)\in\R^2$, we write $x\le y$ if $x_1\le y_1$ and $x_2\le y_2$. Given two points $x,y\in\Z^2$ with $x\leq y$, we define the box $[x,y]=\{z\in\Z^2 | x\leq z\leq y\}$. For $u\in\Z^2$ we denote  $\Z^2_{\geq u}=\{x:x\geq u\}$.

Finally, for $\lambda>0$, $X\sim\mathrm{Exp}(\lambda)$ denotes a random variable $X$ which has exponential distribution with rate $\lambda$, in other words $P(X>t)=e^{-\lambda t}$ for $t\ge 0$, and thus the mean is $\E(X)=\lambda^{-1}$ and variance $\Var(X)=\lambda^{-2}$. To lighten notation, we do not write explicitly the integer parts, as our results are insensitive to shifting points by order $1$. For instance, for a $\xi=(\xi_1,\xi_2)\in\R^2$, $\xi N$ means $(\lfloor \xi_1 N\rfloor,\lfloor \xi_2 N\rfloor)$.

\subsection{Ordering of paths}
We construct two partial orders on directed paths in $\Z^2$.
	\begin{enumerate}
		\item[$\leq$:] For $x,y\in \Z^2$ we write $x \leq y$ if $y$ is above and to the right of $x$, i.e.
		\begin{equation}
		x_1 \leq y_1 \quad \text{and} \quad x_2\leq y_2.
		\end{equation}
		We also write $x < y$ if
		\begin{equation}
		x \leq y \quad \text{and} \quad x\neq y
		\end{equation}
		An up-right path is a (finite or infinite) sequence $\cY=(y_k)_{k}$ in $\Z^2$ such that $y_k-y_{k-1}\in\{\mathrm{e}_1,\mathrm{e}_2\}$ for all $k$. Let $\cU\cR$ be the set of up-right paths in $\Z^2$.
		If $A,B\subset \Z^2$, we write $A \leq B$ if
		\begin{equation}\label{og3}
		x \leq y \quad \forall x\in A\cap\cY,y\in B\cap\cY \quad \forall\cY\in \cU\cR.
		\end{equation}
		where we take the inequality to be vacuously true if one of the intersections in \eqref{og3} is empty.

		\item[$\preceq$:] For $x,y\in \Z^2$ we write $x \preceq y$ if $y$ is below and to the right of $x$, i.e.
		\begin{equation}
		x_1 \leq y_1 \quad \text{and} \quad x_2\geq y_2.
		\end{equation}
		We also write $x\prec y$ if
		\begin{equation}
		x \preceq y \quad \text{and} \quad x\neq y
		\end{equation}
		A down-right path is sequence $\cY=(y_k)_{k\in\Z}$ in $\Z^2$ such that $y_k-y_{k-1}\in\{\mathrm{e}_1,-\mathrm{e}_2\}$ for all $k\in\Z$. Let $\cD\cR$ be the set of infinite down-right paths in $\Z^2$.
		If $A,B\subset \Z^2$, we write $A \preceq B$ if
		\begin{equation}\label{og1}
		x \preceq y \quad \forall x\in A\cap\cY,y\in B\cap\cY \quad \forall\cY\in \cD\cR.
		\end{equation}
		where we take the inequality to be vacuously true if one of the intersections in \eqref{og1} is empty (see Figure~\ref{fig:ogs}).
	\end{enumerate}
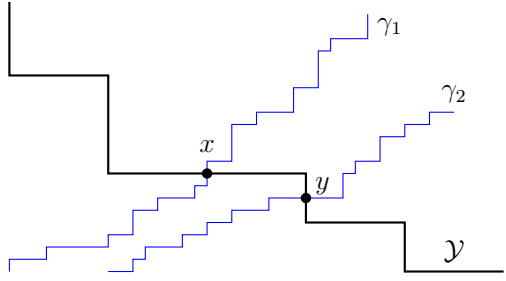
\begin{figure}[t]
\begin{center}
	\begin{tikzpicture}[scale=0.65, every node/.style={transform shape}]
	node/.style={transform shape}]
 \def\s{1.33}
 \draw [thick] (0,5.5) -- (0,4) -- (2,4) -- (2,2) -- (6,2) -- (6,1) -- (8,1) -- (8,0) -- (10,0);
 \node [scale=\s,above] at (9,0) {$\mathcal{Y}$};
 \draw [thin, color=blue] (0,0) -- (0,0.25)-- (0.25,0.25) -- (0.75,0.25) -- (0.75,0.5) -- (2,0.5) -- (2,0.75) -- (2.5,0.75) -- (2.5,1.25) -- (3,1.25) -- (3,1.5) -- (3.5,1.5) -- (3.75,1.5) -- (3.75,1.75) -- (4,1.75) -- (4,2.25) -- (4.5,2.25) -- (4.5,3) -- (5,3) -- (5,3.25) -- (5.75,3.25) -- (5.75,3.75) -- (6.25,3.75) -- (6.25,4.5) -- (6.5,4.5) -- (6.5,4.75) -- (7.25,4.75) -- (7.25,5.25);
 \node [scale=\s,right] at (7.25,5) {$\gamma_1$};
 \fill (4,2) circle (3pt);
 \node [scale=\s,above] at (4,2.25) {$x$};
 \draw [thin, color=blue] (2,0) -- (2.5,0)-- (2.5,0.25) -- (2.75,0.25) -- (2.75,0.5) -- (3.5,0.5) -- (3.5,0.75) -- (4,0.75) -- (4,1) -- (4.5,1) -- (4.5,1.25) -- (5.25,1.25) -- (5.25,1.5) -- (6.75,1.5) -- (6.75,2) -- (7,2) -- (7,2.25) -- (7.5,2.25) -- (7.5,2.75) -- (8,2.75) -- (8,3) -- (8.5,3) -- (8.5,3.25) -- (9,3.25);
 \node [scale=\s,above] at (9,3.25) {$\gamma_2$};
 \fill (6,1.5) circle (3pt);
 \node [scale=\s,right] at (6,1.75) {$y$};
	\end{tikzpicture}
\caption{The two geodesics $\gamma_1$ and $\gamma_2$ are ordered i.e., $\gamma_1\prec \gamma_2$. For any down-right path $\cY$ in $\Z^2$ the set of points $x=\cY\cap \gamma_1$ and $y=\cY\cap \gamma_2$ are ordered, i.e., $x\prec y$.}
\label{fig:ogs}
\end{center}
\end{figure}

\subsection{Stationary LPP}\label{SectStatLPP}
Stationary LPP on $\Z_{\geq 0}^2$ has been introduced in~\cite{PS01} by adding boundary terms on the $\mathrm{e}_1$ and $\mathrm{e}_2$ axis. In~\cite{BCS06} it was shown that it can be set up by using more general boundary domains. In this paper we are going to use two of them.

\paragraph{Boundary weights on the axis.}
For a base point $o=(o_1,o_2)\in \Z^2$ and a parameter value $\rho\in(0,1)$ we introduce the stationary last-passage percolation process $G^\rho_{o,\bullet}$ on $o+\Z_{\ge0}^2$. This process has boundary conditions given by two independent sequences
	\begin{equation}\label{IJ}
	\{I^\rho_{o+i\mathrm{e}_1}\}_{i=1}^{\infty} \quad\text{and}\quad
	\{J^\rho_{o+j\mathrm{e}_2}\}_{j=1}^{\infty}
	\end{equation}
	of i.i.d.\ random variables with $I^\rho_{o+\mathrm{e}_1}\sim\textrm{Exp}(1-\rho)$ and $J^\rho_{o+\mathrm{e}_2}\sim\textrm{Exp}(\rho)$.
 Put $G^\rho_{o,o}=0$ and on the boundaries
	\begin{equation}\label{Gr1} G^\rho_{o,\,o+\,k\mathrm{e}_1}=\sum_{i=1}^k I_{o+i\mathrm{e}_1}
	\quad\text{and}\quad
	G^\rho_{o,\,o+\,l\mathrm{e}_2}= \sum_{j=1}^l J_{o+j\mathrm{e}_2}.
 \end{equation}
	Then in the bulk
	for $x=(x_1,x_2)\in o+ \Z_{>0}^2$,
	\begin{equation}\label{Gr2}
	G^\rho_{o,\,x}= \max_{1\le k\le x_1-o_1} \; \Bigl\{ \;\sum_{i=1}^k I_{o+i\mathrm{e}_1} + G_{o+k\mathrm{e}_1+\mathrm{e}_2, \,x} \Bigr\}
	\bigvee
	\max_{1\le \ell\le x_2-o_2}\; \Bigl\{ \;\sum_{j=1}^\ell J_{o+j\mathrm{e}_2} + G_{o+\ell \mathrm{e}_2+\mathrm{e}_1, \,x} \Bigr\} .
	\end{equation}

\paragraph{Boundary weights on antidiagonal.}
The stationary model with density $\rho$ can be realized by putting boundary weights on $\cL$ as follows. Let $\{X_k,k\in\Z\}$ and $\{Y_k,k\in\Z\}$ be independent random variables with $X_k\sim\textrm{Exp}(1-\rho)$ and $Y_k\sim\textrm{Exp}(\rho)$. Then, define
\begin{equation}
h_0(k,-k)=\left\{
\begin{array}{ll}
\sum_{\ell=1}^k(X_\ell-Y_\ell),&\textrm{ for }k\geq 1,\\
0,&\textrm{ for }k=0,\\
-\sum_{\ell=k+1}^0 (X_\ell-Y_\ell),&\textrm{ for }k\leq -1.
\end{array}\right.
\end{equation}
Then the LPP defined by \eqref{eq1.2} with initial condition $h_0$ is stationary, that is, the increments $G^{h_0}_{\cL,x+\mathrm{e_1}}-G^{h_0}_{\cL,x}\sim\textrm{Exp}(1-\rho)$ as well as $G^{h_0}_{\cL,x+\mathrm{e_2}}-G^{h_0}_{\cL,x}\sim\textrm{Exp}(\rho)$ for all $x>\cL$.

 Next we define the \emph{exit points} of geodesics, these will play an important role in our analysis.
\begin{definition}[Exit points]\label{DefExitPoints}$ $ \\
(a) For a point $p\in o+\Z_{>0}^2$, let $Z_{o,p}$ be the signed exit point of the geodesic $\pi_{o,p}$ of $G_{o,p}$ from the west and south boundaries of $o+\Z_{>0}^2$. More precisely,
	\begin{equation}\label{exit2}
	Z^\rho_{o,p}=
	\begin{cases}
	\argmax{k} \bigl\{ \,\sum_{i=1}^k I_{o+i\mathrm{e}_1} + G_{o+k\mathrm{e}_1+\mathrm{e}_2, \,x} \bigr\} &\textrm{if } \pi_{o,p}\cap o+\mathrm{e}_1\neq\emptyset,\\
	-\argmax{\ell}\bigl\{ \;\sum_{j=1}^\ell J_{o+j\mathrm{e}_2} + G_{o+\ell \mathrm{e}_2+\mathrm{e}_1, \,x} \bigr\} &\textrm{if } \pi_{o,p}\cap o+\mathrm{e}_2\neq\emptyset.
	\end{cases}
	\end{equation}
(b) For a point $p>\cL$, we denote by $Z^{h_0}_{\cL,p}\in\Z$ the exit point of the LPP from $\cL$ with initial condition $h_0$, if the starting point of the geodesic from $\cL$ to $p$ is given by $(Z^{h_0}_{\cL,p},-Z^{h_0}_{\cL,p})$. In the case of the stationary model with parameter $\rho$, the exit point is denoted by $Z^{\rho}_{\cL,p}$.
\end{definition}

	As a consequence, the value $G^\rho_{o,x}$ can be determined by \eqref{Gr1} and the recursion
	\begin{equation}\label{recu}
	G^\rho_{o,x}=\omega_x+G^\rho_{o,x-\mathrm{e}_1}\vee G^\rho_{o,x-\mathrm{e}_2}.
	\end{equation}

	\subsection{Backward LPP}
	Next we consider LPP maximizing down-left paths. For $y\leq o$, define
	\begin{equation}\label{v:Gr}
	\widehat{G}_{o,y}=G_{y,o},
	\end{equation}
	and let the associated geodesic be denoted by $\hat{\pi}_{o,y}$. For each $o=(o_1,o_2)\in\Z^2$ and a parameter value $\rho\in(0,1)$ define a stationary last-passage percolation processes $\widehat{G}^\rho$ on $o+\Z^2_{\le0}$, with boundary variables on the north and east, in the following way. Let
	\begin{equation}\label{IJ hat}
	\{\hat{I}^\rho_{o-i\mathrm{e}_1}\}_{i=1}^{\infty}
	\quad\text{and}\quad
	\{\hat{J}^\rho_{o-j\mathrm{e}_2}\}_{j=1}^{\infty}
	\end{equation}
	be two independent sequences of i.i.d.\ random variables with marginal distributions $\hat{I}^\rho_{o-i\mathrm{e}_1}\sim\textrm{Exp}(1-\rho)$ and $\hat{J}^\rho_{o-j\mathrm{e}_2}\sim\textrm{Exp}(\rho)$. The boundary variables in \eqref{IJ} and those in \eqref{IJ hat} are taken independent of each other. Put $\widehat{G}^\rho_{o,\,o}=0$ and on the boundaries
	\begin{equation}\label{Gr11} \widehat{G}^\rho_{o,\,o-k\mathrm{e}_1}=\sum_{i=1}^k \hat{I}_{o-i\mathrm{e}_1}
	\quad\text{and}\quad
	\widehat{G}^\rho_{o,\,o-l\mathrm{e}_2}= \sum_{j=1}^l \hat{J}_{o-j\mathrm{e}_2}. \end{equation}
	Then in the bulk
	for $x=(x_1,x_2)\in o+ \Z_{<0}^2$,
	\begin{equation}\label{Gr12}
	\widehat{G}^\rho_{o,\,x}= \max_{1\le k\le o_1-x_1} \; \Bigl\{ \;\sum_{i=1}^k \hat{I}_{o-i\mathrm{e}_1} + \widehat{G}_{ \,o-k\mathrm{e}_1-\mathrm{e}_2,x} \Bigr\}
	\bigvee
	\max_{1\le \ell\le o_2-x_2}\; \Bigl\{ \;\sum_{j=1}^\ell \hat{J}_{o-j\mathrm{e}_2} + \widehat{G}_{ \,o-\ell \mathrm{e}_2-\mathrm{e}_1,x} \Bigr\} .
	\end{equation}
	For a southwest endpoint $p\in o+\Z_{<0}^2$, let $\widehat{Z}^\rho_{o,p}$ be the signed exit point of the geodesic $\hat{\pi}_{o,p}$ of $\widehat{G}^\rho_{o,p}$ from the north and east boundaries of $o+\Z_{<0}^2$. Precisely,
	\begin{equation}\label{exit}
	\widehat{Z}^\rho_{o,\,x}=
	\begin{cases}
	\argmax{k} \bigl\{ \,\sum_{i=1}^k \hat{I}_{o-i\mathrm{e}_1} + \widehat{G}_{\,o-k\mathrm{e}_1-\mathrm{e}_2,x} \bigr\}, &\text{if } \hat{\pi}_{o,x}\cap o-\mathrm{e}_1\neq\emptyset,\\
	-\argmax{\ell}\bigl\{ \;\sum_{j=1}^\ell \hat{J}_{o-j\mathrm{e}_2} + \widehat{G}_{\,o-\ell \mathrm{e}_2-\mathrm{e}_1,x} \bigr\}, &\text{if } \hat{\pi}_{o,x}\cap o-\mathrm{e}_2\neq\emptyset.
	\end{cases}
	\end{equation}

\subsection{Comparison Lemma}\label{SectCompLemma}
We are going to use the comparison between point-to-point LPP and stationary LPP using the lemma by Cator and Pimentel.
\begin{lemma}\label{Lem:Comparison}
Let $o=(0,0)$ and consider two points $p^1\preceq p^2$.\\
If $Z^\rho_{o,p^1}\geq 0$, then
\begin{equation}
 G_{o,p^2}-G_{o,p^1}\leq G^\rho_{o,p^2}-G^\rho_{o,p^1}.
\end{equation}
 If $Z^\rho_{o,p^2}\leq 0$, then
\begin{equation}
 G_{o,p^2}-G_{o,p^1}\geq G^\rho_{o,p^2}-G^\rho_{o,p^1}.
\end{equation}
\end{lemma}
Clearly by reversion of the space we can use this comparison lemma also for backwards LPP.

\begin{lemma}\label{Lem:ComparisonB}
Consider two points $p^1\preceq p^2$ with $p^1,p^2>\cL$.\\
If $Z^\rho_{\cL,p^1}\geq Z^{h_0}_{\cL,p^2}$, then
\begin{equation}
 G^{h_0}_{\cL,p^2}-G^{h_0}_{\cL,p^1}\leq G^\rho_{\cL,p^2}-G^\rho_{\cL,p^1}.
\end{equation}
 If $Z^\rho_{\cL,p^2}\leq Z^{h_0}_{\cL,p^1}$, then
\begin{equation}
 G^{h_0}_{\cL,p^2}-G^{h_0}_{\cL,p^1}\geq G^\rho_{\cL,p^2}-G^\rho_{\cL,p^1}.
\end{equation}
\end{lemma}
For $p^1_2=p^2_2$, Lemma~\ref{Lem:Comparison} is proven as Lemma~1 of~\cite{CP15b}, while Lemma~\ref{Lem:ComparisonB} is Lemma~2.1 of~\cite{Pim18}. The generalization to the case of geodesics starting from $\cL$ (or from any down-right paths) is straightforward, see e.g.\ Lemma~3.5 of~\cite{FGN17}.

\section{Local Stationarity}\label{SectLocalStat}

\subsection{Localization over a time-span $\tau N$.}
In this section we are going to prove Theorem~\ref{thm:locCylinder}.
We shall need the following estimates on the tail of the exit point of a stationary process. For a density $\rho\in (0,1)$ we associate a direction
\begin{equation}
\xi(\rho)=\left(\frac{(1-\rho)^2}{(1-\rho)^2+\rho^2},\frac{\rho^2}{(1-\rho)^2+\rho^2}\right)
\end{equation}
and, vice versa, to each direction $\xi=(\xi_1,\xi_2)$ corresponds a density
\begin{equation}
 \rho(\xi)=\frac{\sqrt{\xi_2}}{\sqrt{\xi_1}+\sqrt{\xi_2}}.
\end{equation}

\begin{lemma}[Theorem 2.5 and Proposition 2.7 of \cite{EJS20}]\label{spc}
 Let $\e\in (0,1]$. Then there exists $N_0(\e)$, $c_0(\e)$, $r_0(\e)>0$ such that for every direction $\xi$ with $\e\leq \xi_2/\xi_1\leq 1/\e$, $N\geq N_0$ and $r\geq r_0$:
	\begin{align}
		\P(|Z^{\nu}_{o,\xi N}|> rN^{2/3})&\leq e^{-c_0r^3},\label{spc1}\\
		\P(Z^{\nu}_{o,\xi N-rN^{2/3}\mathrm{e}_3}> 0)&\leq e^{-c_0r^3},\label{spc2}\\
		\P(Z^{\nu}_{o,\xi N+rN^{2/3}\mathrm{e}_3}< 0)&\leq e^{-c_0r^3},\label{spc3}
	\end{align}
for all densities $\nu$ satisfying $|\nu-\rho(\xi)|\leq N^{-1/3}$.
\end{lemma}

Using Lemma~\ref{spc} one can control the path of a point-to-point geodesic.
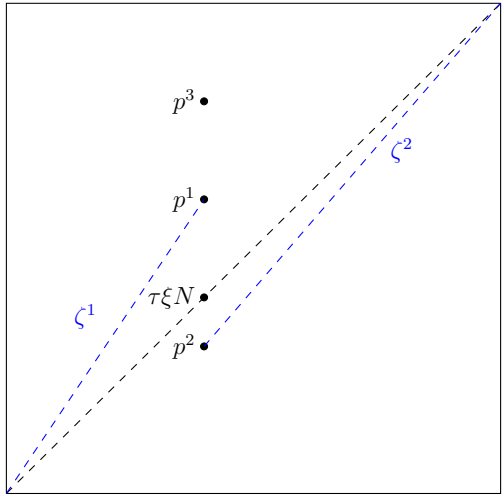
\begin{figure}
	\begin{center}
		\begin{tikzpicture}[scale=0.65, every node/.style={transform shape}]
		node/.style={transform shape}]
		\def\s{1.2}
		\draw  (0,0) -- (0,10) -- (10,10)--(10,0)--(0,0);
		\draw [dashed](0,0)--(10,10);
		\filldraw (4,4) circle (2pt);
		\node [scale=\s,left] at (4,4) {$\tau\xi N$};
		\filldraw (4,6) circle (2pt);
		\node [scale=\s,left] at (4,6) {$p^1$};
		\draw [dashed,blue] (0,0)--(4,6);
		\node [scale=\s,left,blue] at (2,3.6) {$\zeta^1$};
		\filldraw (4,8) circle (2pt);
		\node [scale=\s,left] at (4,8) {$p^3$};
		\filldraw (4,3) circle (2pt);
		\node [scale=\s,left] at (4,3) {$p^2$};
		\draw [dashed,blue] (4,3)--(10,10);
		\node [scale=\s,blue] at (8,7) {$\zeta^2$};
		\end{tikzpicture}
		\caption{Illustration of the geometry of the points and characteristics that appear in the proof of Theorem \ref{thm:loc}.}\label{fig:hfig}
	\end{center}
\end{figure}
\begin{theorem}\label{thm:loc}
Let $\e\in (0,1]$. Then there exist $N_0(\e)$ and $c_1(\e)$ such that for $\xi$ satisfying $\e\leq \xi_2/\xi_1\leq 1/\e$,
	\begin{equation}
	\P\big(\Gamma_{ \tau \xi_1 N}(\pi_{o,\xi N})>M(\tau N)^{2/3} \big)\leq e^{-c_1M^3}
	\end{equation}
for all $\tau N\geq N_0$ and all $M\leq (\tau N)^{1/3}/\log(N)$.
\end{theorem}
\begin{proof}
We will show in detail that
\begin{equation}
	\P\Big(\Gamma^u_{\tau \xi_1 N}(\pi_{o^1,o^2})>\tau \xi_2N +M (\tau N)^{2/3}\Big)\leq e^{-c_1M^3},
\end{equation}
where $o^1=o$ and $o^2=\xi N$. Similarly one proves
\begin{equation}
	\P\Big(\Gamma^l_{\tau \xi_1 N}(\pi_{o^1,o^2})<\tau \xi_2N -M (\tau N)^{2/3}\Big)\leq e^{-c_1M^3}.
\end{equation}
Then Theorem~\ref{thm:loc} follows directly from the definition of $\Gamma_{\tau \xi_1 N}(\pi_{o,\xi N})$. Set the points (see Figure \ref{fig:hfig})
\begin{equation}
\begin{aligned}
	p^1&=\tau N \xi+\tfrac M4(\tau N)^{2/3}\mathrm{e}_2,\\
	p^2&=\tau N \xi-\tfrac M8 \tfrac{1-\tau}{\tau} (\tau N)^{2/3}\mathrm{e}_2,\\
	p^3&=\tau N \xi+\tfrac M2(\tau N)^{2/3}\mathrm{e}_2,
\end{aligned}
\end{equation}
and the characteristics associated with $(o^1,p^1)$ and $(o^2,p^2)$
\begin{equation}
\begin{aligned}
	\zeta^1&=(\tau \xi_1 N, \tau \xi_2 N + \tfrac M4 (\tau N)^{2/3}),\\
	\zeta^2&=((1-\tau) \xi_1 N, (1-\tau) \xi_2 N + \tfrac M8 \tfrac{1-\tau}{\tau} (\tau N)^{2/3}).
\end{aligned}
\end{equation}
The associated densities are
\begin{equation}
\begin{aligned}\label{den}
	\rho_1&=\frac{\sqrt{\tau \xi_2 N+\tfrac M4 (\tau N)^{2/3}}}{\sqrt{\tau \xi_1 N}+\sqrt{\tau \xi_2 N+\tfrac M4 (\tau N)^{2/3}}},\\
	\rho_2&=\frac{\sqrt{(1-\tau) \xi_2 N +\tfrac M8 \tfrac{1-\tau}{\tau} (\tau N)^{2/3}}}{\sqrt{(1-\tau) \xi_1 N}+\sqrt{(1-\tau) \xi_2 N + \tfrac M8 \tfrac{1-\tau}{\tau} (\tau N)^{2/3}}}.
\end{aligned}
\end{equation}
Note that by \eqref{spc2}--\eqref{spc3} there exists $c_0>0$ such that
\begin{equation}
\begin{aligned}\label{zub}
	\P\big(Z^{\rho_1}_{o^1,p^3}>0\big)\leq e^{-c_0M^{3}},\\
	\P\big(\hat{Z}^{\rho_2}_{o^2,p^3}<0\big)\leq e^{-c_0M^{3}}.
\end{aligned}
\end{equation}
Define, for $i\geq 0$,
\begin{equation}
\begin{aligned}
	J_i&=G_{o^1,p^3+(i+1)\mathrm{e}_2}-G_{o^1,p^3+i\mathrm{e}_2},\\
	\widehat{J}_i&=G_{o^2,p^3+\mathrm{e}_1+i\mathrm{e}_2}-G_{o^2,p^3+\mathrm{e}_1+(i+1)\mathrm{e}_2},\\
	J^{\rho_1}_i&=G^{\rho_1}_{o^1,p^3+(i+1)\mathrm{e}_2}-G^{\rho_1}_{o^1,p^3+i\mathrm{e}_2},\\
	\widehat{J}^{\rho_2}_i&=G^{\rho_2}_{o^2,p^3+\mathrm{e}_1+i\mathrm{e}_2}-G^{\rho_2}_{o^2,p^3+\mathrm{e}_1+(i+1)\mathrm{e}_2}.
\end{aligned}
\end{equation}
Then, by the Lemma~\ref{Lem:Comparison}, it follows from \eqref{zub} that with probability $1-2e^{-c_0M^{3}}$
\begin{equation}
	J_i\leq J_i^{\rho_1} \textrm{ and }\widehat{J}_i^{\rho_2} \leq \widehat{J}_i
\end{equation}
for all $i\geq 0$, and therefore that
\begin{equation}
	J_i-\widehat{J}_i\leq J_i^{\rho_1}-\widehat{J}_i^{\rho_2}\label{rwc}
\end{equation}
for all $i\geq 0$. Set $\rho=\rho(\xi)$. Note that for $M\leq (\tau N)^{1/3}/\log(N)$ , using series expansion we have
\begin{equation}
\begin{aligned}\label{rhod}
	\rho_1&=\rho+\kappa(\rho)\frac M8(\tau N)^{-1/3}+o((\tau N)^{-1/3}),\\
	\rho_2&=\rho+\kappa(\rho)\frac M{16}(\tau N)^{-1/3}+o((\tau N)^{-1/3}),\\
	\rho_1-\rho_2&=\kappa(\rho)\frac M{16}(\tau N)^{-1/3}+o((\tau N)^{-1/3})>0,
\end{aligned}
\end{equation}
with $\kappa(\rho)=(1-\rho)(1-2\rho(1-\rho))/\rho>0$ for all $\rho\in (0,1)$.

Define
\begin{equation}
S_i=\sum_{k=0}^{i}J_k-\widehat{J}_k\quad \textrm{and}\quad W_i=\sum_{k=0}^{i}J^{\rho_1}_k-\widehat{J}^{\rho_2}_k
\end{equation}
so that by \eqref{rwc}
\begin{equation}
	S_i\leq W_i\textrm{ for }i\geq 0.
\end{equation}
Note that
\begin{equation}
	\{\Gamma^u_{\tau \xi_1 N}(\pi_{o^1,o^2})>\tau \xi_2 N +M (\tau N)^{2/3}\} \subseteq \bigg\{\sup_{i\geq \tfrac M2(\tau N)^{2/3}} S_i>0\bigg\} \subseteq \bigg\{\sup_{i\geq \tfrac{M}{2} (\tau N)^{2/3}} W_i>0\bigg\}.
\end{equation}
It follows that it is enough to show that there exists $c_1>0$ such that
\begin{equation}\label{ub5}
	\P\bigg(\sup_{i\geq \tfrac M2(\tau N)^{2/3}} W_i>0\bigg) \leq e^{-c_1M^{3}}.
\end{equation}

Note that
\begin{equation}\label{ub4}
\begin{aligned}
 	\P\bigg(\sup_{i\geq \tfrac M2 (\tau N)^{2/3}} W_i>0\bigg)& \leq \P\Big(W_{\tfrac M2 (\tau N)^{2/3}}>-\frac{ \chi(\rho)M^2}{64}(\tau N)^{1/3}\Big)\\ &+\P\bigg(\sup_{i\geq \tfrac M2 (\tau N)^{2/3}} W_i-W_{\tfrac M2 (\tau N)^{2/3}}>\frac{ \chi(\rho)M^2}{64}(\tau N)^{1/3}\bigg)
\end{aligned}
\end{equation}
for $\chi(\rho)=\kappa(\rho)/\rho^2$.

Plugging \eqref{rhod} in Lemma~\ref{lem:crw}
\begin{equation}
\begin{aligned}\label{ub3}
	\P\bigg(\sup_{i\geq \tfrac M2 (\tau N)^{2/3}} W_i-W_{\tfrac M2 (\tau N)^{2/3}}>\frac{ \chi(\rho)M^2}{64}(\tau N)^{1/3}\bigg)
&\leq \frac{\rho_1}{\rho_2}e^{-(\rho_1-\rho_2)\frac {\chi(\rho)M^2}{64 }(\tau N)^{1/3}}\\
&\leq 2e^{-\chi(\rho) \kappa(\rho) M^3/1024}
\end{aligned}
\end{equation}
for all $\tau N$ large enough.

Next, using exponential Tchebishev inequality, we show that
\begin{equation}\label{eq3.20}
 \P\Big(W_{\tfrac M2(\tau N)^{2/3}}> -\frac{ \chi(\rho)M^2}{64}(\tau N)^{1/3}\Big)\leq 2 e^{-M^3 \chi(\rho)\kappa(\rho)/8192}
\end{equation}
for all $\tau N$ large enough, which completes the proof. Indeed, using
\begin{equation}
\left(\frac{1}{\rho_1}-\frac{1}{\rho_2}\right)\tfrac M2(\tau N)^{2/3} = -\frac{M^2\chi(\rho)}{32}(\tau N)^{1/3}+o((\tau N)^{1/3}),
\end{equation}
we get, using also the independence of the $J$'s and $\widehat J$'s, that
\begin{equation}
\begin{aligned}
\eqref{eq3.20} &= \P\Bigg(\sum_{k=0}^{\tfrac M2(\tau N)^{2/3}}(J_k^{\rho_1}-\rho_1^{-1}-\widehat{J}_k^{\rho_2}+\rho_2^{-1})> \frac{ \chi(\rho)M^2}{64}(\tau N)^{1/3}+o((\tau N)^{1/3})\Bigg)\\
&\leq \inf_{\lambda>0} \frac{\E\Big(e^{\lambda (J_1^{\rho_1}-\rho_1^{-1}-\widehat{J}_1^{\rho_2}+\rho_2^{-1})}\Big)^{\tfrac M2(\tau N)^{2/3}}}{e^{\lambda [\frac{ \chi(\rho)M^2}{64}(\tau N)^{1/3}+o((\tau N)^{1/3})]}}\\
&=\inf_{\mu>0} e^{-M(M\kappa(\rho)-32 \mu)\mu/(64\rho^2)+o(1)} \leq 2 e^{-M^3\kappa(\rho)\chi(\rho)/8192},
\end{aligned}
\end{equation}
for all $\tau N$ large enough, where in the third step we set $\lambda=\mu (\tau N)^{-1/3}$ and performed simple computations.

\end{proof}

\begin{theorem}\label{prop1}
Let $o=(0,0)$ and $\e\in (0,1]$. There there exists $N_1(\e)$, $c(\e)$, $C(\e)$ such that for every direction $\xi$ with $\e\leq \xi_2/\xi_1\leq 1/\e$, and $v\leq N^{1/3}/\log(N)$, for $N>N_1$
	\begin{equation}
	\P\Big(\max_{k\in[0, \xi_1N]}\Gamma_k(\pi_{o,\xi N})< v N^{2/3}\Big)\geq 1 - Ce^{-c v^3}.
	\end{equation}
\end{theorem}
\begin{proof}
	The proof follows the approach of \cite{BSS14}, using the pointwise control of the fluctuations of the geodesic around the characteristic from Theorem~\ref{thm:loc}. Let $m=\min\{j: 2^{-j}N \leq N^{1/2}\}$. Choose $u_1<u_2<\ldots$ with $u_1=v/10$ and $u_j-u_{j-1}=u_1 2^{-(j-1)/2}$. We define
\begin{equation}
	u(k)=\Gamma^u_k(\pi_{o,\xi N})-\tfrac{\xi_2}{\xi_1} k, \quad k\in [0,\xi_1 N]
\end{equation}
and the following events
	\begin{equation}
	\begin{aligned}
	A_j&=\{u(k 2^{-j} N) \leq u_j N^{2/3}, 1\leq k \leq 2^j-1\},\\
	B_{j,k}&=\{u(k 2^{-j} N) > u_j N^{2/3}\}, \quad k=1,\ldots,2^j-1,\\
	L&=\{\sup_{x\in[0,1]} |u((k+x) 2^{-m} N)-u(k 2^{-m} N)|\leq \tfrac12 v N^{2/3}, 0\leq k \leq 2^m-1\},\\
	G&=\{u(k)\leq v N^{2/3}\textrm{ for all }0\leq k\leq \xi_1N\}.
	\end{aligned}
	\end{equation}
	Notice that $A_j^c=\bigcup_{k=1}^{2^j-1} B_{j,k}$. Also, since $\lim_{j\to\infty} u_j\leq v/2$, we have
	\begin{equation}
	\bigcup_{j=1}^m\bigcup_{k=1}^{2^j-1} (B_{j,k}\cap A_{j-1})\supseteq \{u(k 2^{-m} N)\geq \tfrac12v N^{2/3}\textrm{ for some }k=1,\ldots,2^m-1\}.
	\end{equation}
	This implies that
	\begin{equation}
	G\supseteq \Bigg(\bigcup_{j=1}^m\bigcup_{k=1}^{2^j-1} (B_{j,k}\cap A_{j-1})\Bigg)^c\cap L.
	\end{equation}
	Thus we have
	\begin{equation}\label{eq2}
	\P(G)\leq \P(L^c)+\sum_{j=1}^m\sum_{k=1}^{2^j-1}\P(B_{j,k}\cap A_{j-1}).
	\end{equation}

	Since the geodesics have discrete steps, in $n$ time steps a geodesic can wonder off by at most $n$ steps from its characteristic. For all $N$ large enough, $N^{1/2}<\tfrac12 v N^{2/3}$ and therefore $\P(L)=1$. Thus we need to bound $\P(B_{j,k}\cap A_{j-1})$ only. As for even $k$ the two events are incompatible, we consider odd $k$.
	
	If $A_{j-1}$ holds, then the geodesic at $t_1=(k-1)2^{-j}N$ and $t_2=(k+1) 2^{-j}N$ satisfies
	\begin{equation}
	 	u(t_1)\leq u_{j-1} N^{2/3}\quad \text{ and } \quad u(t_2)\leq u_{j-1} N^{2/3}.
	\end{equation}
	 Consider the point-to-point LPP from $\hat o^1$ to $\hat o^2$ with
	\begin{equation}
	 \hat o^1=(t_1,t_1\tfrac{\xi_2}{\xi_1}+u_{j-1}) \quad \text{and} \quad \hat o^2=(t_2, t_2 \tfrac{\xi_2}{\xi_1}+u_{j-1}).
\end{equation}
Let $\hat u(i)=\Gamma^u_i(\pi_{\hat o^1,\hat o^2})$ for $i\in[t_1,t_2]$. Then, by the order of geodesics
	\begin{equation}
	 u(i)\leq \hat u(i)\textrm{ for } i\in[t_1,t_2],
\end{equation}
	 so that
	\begin{equation}
	 	\{u(i)>u_jN^{2/3}\} \subseteq \{\hat u(i)>u_jN^{2/3}\}\textrm{ for } i\in[t_1,t_2].
\end{equation}
This gives
\begin{equation}
\P(B_{j,k}\cap A_{j-1}) \leq \P\big(\hat u(k 2^{-j}N)>u_jN^{2/3}).
\end{equation}
Since the law of $\hat u$ is the one of a point-to-point LPP over a time distance $t_2-t_1=2^{-j+1}N$, we can apply Theorem~\ref{thm:loc} with $\tau=1/2$, $N=t_2-t_1$ $M$ satisfying $(u_{j}-u_{j-1}) N^{2/3}=M (\tfrac12(t_j-t_{j-1}))^{2/3}$. This gives
\begin{equation}
\P\big(\hat u(k 2^{-j}N)>u_jN^{2/3}) \leq e^{-c_1 (u_1 2^{-(j-1)/2} 2^{2j/3})^3}\leq e^{-c_1 u_1^3 2^{j/2}}.
\end{equation}
This bound applied to \eqref{eq2} leads to $\P(G)\leq C e^{-c v^3}$ for some constants $C,c>0$.
\end{proof}

Now we have all the ingredients to prove Theorem~\ref{thm:locCylinder}.
\begin{proof}[Proof of Theorem~\ref{thm:locCylinder}]
Theorem~\ref{thm:loc} implies that with probability at least $1-e^{-c_1 M^3/8}$, the geodesic from $o$ to $\xi N$ does not deviate more than $\tfrac12 M (\tau N)^{2/3}$ away from the point $\tau \xi N$. Given this event, by order of geodesics, the geodesic from $o$ to $\xi N$ is sandwiched between the geodesics from $\tfrac12 M (\tau N)^{2/3} \mathrm{e}_1$ to $\tau \xi N+\tfrac12 M (\tau N)^{2/3} \mathrm{e}_1$ and the one from $-\tfrac12 M(\tau N)^{2/3}\mathrm{e}_1$ to $\tau \xi N-\tfrac12 M (\tau N)^{2/3} \mathrm{e}_1$. By Theorem~\ref{prop1}, the latter two geodesics fluctuates no more than $\tfrac12 M (\tau N)^{2/3}$, with probability at least $1-C e^{-cM^3/\tau^2}$, which implies the claim.
\end{proof}

\subsection{Localization of the exit point}
In this subsection, we estimate the location of the exit point for densities slightly larger or smaller than $1/2$. This will allow us to sandwich the point-to-point geodesics by those of the stationary. Notice that to apply Lemma~\ref{Lem:Comparison}, it would be enough to set in the event $\cA_1$ (resp.\ $\cA_2$) below that the exit point is positive (resp.\ negative) and bounded by $15 r N^{2/3}$ (resp.\ $-15 r N^{2/3}$) as the exit point for the LPP $G_{o,x}$ is $0$. However, with this slight modification (that the exit point is $rN^{2/3}$ from the origin), the proof is then applicable for more general initial conditions provided the exit points of the LPP with initial conditions $h_0$ on $\cL$ is localized in a $[-r N^{2/3},r N^{2/3}]$ with high probability.

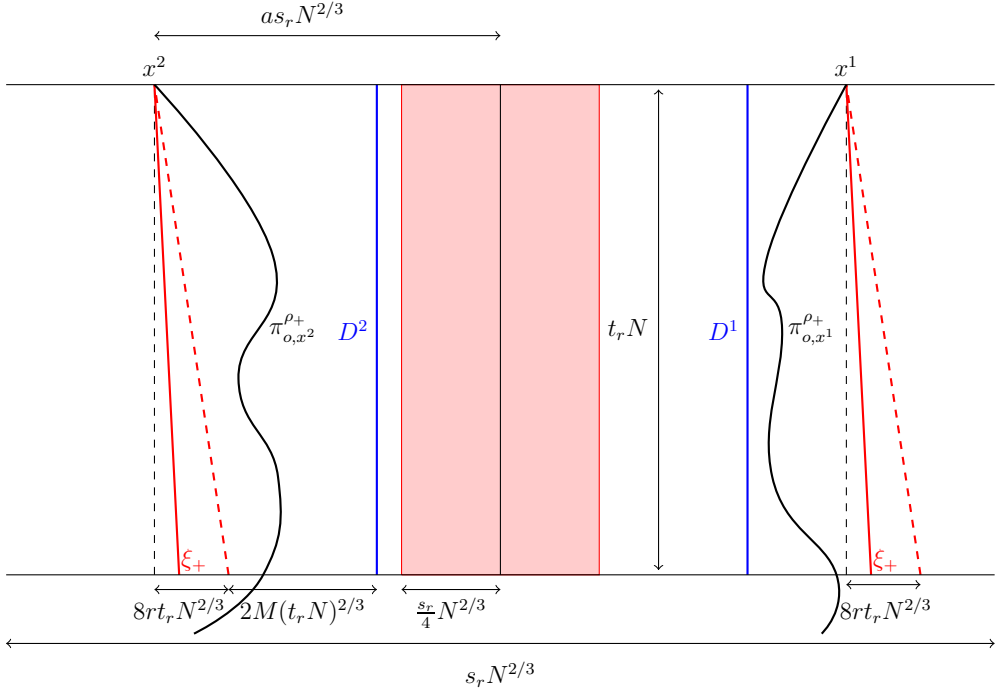
\begin{figure}
\begin{center}
	\begin{tikzpicture}[scale=0.65, every node/.style={transform shape}]
	node/.style={transform shape}]
	\def\s{1.2}
	\draw [color=red,fill=red!20] (8,0) -- (12,0) -- (12,10)--(8,10) --(8,0);
	\node [scale=\s,above] at (3,10) {$x^2$};
	\node [scale=\s,above] at (17,10) {$x^1$};
	\draw [dashed] (3,0)--(3,10);
	\draw [dashed,red,thick] (4.5,0)--(3,10);
	\draw [red,thick] (3.5,0)--(3,10);
	\node [scale=\s,red] at (3.8,0.3) {$\xi_+$};
	\draw [dashed] (17,0)--(17,10);
	\draw [dashed,red,thick] (18.5,0)--(17,10);
	\draw [red,thick] (17.5,0)--(17,10);
	\node [scale=\s,red] at (17.8,0.3) {$\xi_+$};
	\draw [blue,thick] (7.5,10)--(7.5,0);
	\node [scale=\s,left,blue] at (7.5,5) {$D^2$};
	\draw [blue,thick] (15,10)--(15,0);
	\node [scale=\s,left,blue] at (15,5) {$D^1$};
	\draw [<->] (3,11)--(10,11);
	\node [scale=\s,above] at (6,11) {$as_r N^{2/3}$};
	\draw [<->,right] (13.2,9.9)--(13.2,0.1);
	\node [scale=\s,right] at (12,5) {$t_r N$};
	\draw [<->] (8,-0.3)--(10,-0.3);
	\node [scale=\s,below] at (9,-0.3) {$\frac{s_r}4 N^{2/3}$};
	\draw [<->] (3,-0.3)--(4.5,-0.3);
	\node [scale=\s,below] at (3.5,-0.3) {$8rt_r N^{2/3}$};
	\draw [<->] (17,-0.2)--(18.5,-0.2);
	\node [scale=\s,below] at (17.8,-0.3) {$8rt_r N^{2/3}$};
	\draw [<->] (4.5,-0.3)--(7.5,-0.3);
	\node [scale=\s,below] at (6,-0.3) {$2M(t_r N)^{2/3}$};
	\draw [<->] (0,-1.4)--(20,-1.4);
	\node [scale=\s,below] at (10,-1.6) {$s_r N^{2/3}$};
	\draw[black,thick] plot [smooth,tension=0.8] coordinates{(3.8,-1.2)(5.2,0) (5.5,2) (4.7,4) (5.4,6.5) (3,10)};
	\draw[black,thick] plot [smooth,tension=0.8] coordinates{(16.5,-1.2)(16.8,0) (15.5,2) (15.7,5) (15.4,6.5) (17,10)};
	\node [scale=\s,black] at (5.8,5) {$\pi^{\rho_+}_{o,x^2}$};
	\node [scale=\s,black] at (16.3,5) {$\pi^{\rho_+}_{o,x^1}$};
	\draw (0,10)--(20,10);
	\draw (0,0)--(20,0);
 \draw (10,0)--(10,10);
	\end{tikzpicture}
	\caption{Illustration of the geometry around the end-point $(N,N)$ magnified and rotated by $\pi/4$. Choosing $t_r,s_r$ properly forces the geodesic $\pi^{\rho_+}_{o,x^2}$ to traverse to the left of $D^2$ and $\pi^{\rho_+}_{o,x^1}$ to the right of $D^1$.}\label{fig:dim}
\end{center}
\end{figure}

Fix $r>0$ and let $0<s_r,t_r$. $s_r$ and $t_r$ will be determined later and represent the dimensions of space and time respectively. Let $0<a<1$. Define the points
\begin{equation}
\begin{aligned}\label{ep}
	x^1&=N\mathrm{e}_4 + a s_r N^{2/3} \mathrm{e}_3,\\
	x^2&=N\mathrm{e}_4 - a s_r N^{2/3} \mathrm{e}_3.
\end{aligned}
\end{equation}
Define the densities
\begin{equation}
\rho_+ = \tfrac12 + r N^{-1/3},\quad \rho_- = \tfrac12 - r N^{-1/3},
\end{equation}
and, with $o=(0,0)$, the events
\begin{equation}\label{A}
\begin{aligned}
	\cA_1&=\big\{Z^{\rho_+}_{o,x^2}\geq rN^{2/3} ,Z^{\rho_+}_{o,x^1}\leq 15rN^{2/3}\big\},\\
	\cA_2&=\big\{Z^{\rho_-}_{o,x^2}\geq -15rN^{2/3} ,Z^{\rho_-}_{o,x^1}\leq -rN^{2/3}\big\},\\
	\cA&=\cA_1\cup\cA_2.
\end{aligned}
\end{equation}

The event $\cA$ is highly probable for large $r$ as shows the following lemma.
\begin{lemma}\label{lem:sb}
Assume $0\leq s_r\leq 2r$ and $0<a<1$. There exists $c,N_0>0$ such that for $N>N_0$ and $0<r<N^{1/3}/\log(N),$	
	\begin{equation}
		\P(\cA)\geq 1-e^{-cr^3}.
	\end{equation}
\end{lemma}
\begin{proof}
	We will show the claim for $\cA_1$, one can similarly prove the claim for $\cA_2$. The result would then follow from union bound. To prove the claim for $\cA_1$, by union bound it is enough to show that
	\begin{align}
		&\P\big(Z^{\rho_+}_{o,x^2}\geq rN^{2/3}\big)\geq 1-e^{-cr^3}\label{lb1},\\
		&\P\big(Z^{\rho_+}_{o,x^1}\leq 15rN^{2/3}\big)\geq 1-e^{-cr^3}\label{lb2}.
	\end{align}
Let $x^3=x^2-r N^{2/3} \mathrm{e}_1$. Then by stationarity of the model,
\begin{equation}\label{eq3.44}
 \P\big(Z^{\rho_+}_{o,x^2}< rN^{2/3}\big) = \P\big(Z^{\rho_+}_{o,x^3}< 0\big).
\end{equation}
Now we want to use \eqref{spc3}. For that denote $\tilde N=N-\frac{r}{2} N^{2/3}$ and write
\begin{equation}
 x^3=\xi(\rho_+) 2\tilde N+\tilde r \tilde N^{2/3} \mathrm{e_3}.
\end{equation}
Solving with respect to $\tilde r$ we obtain
\begin{equation}
 \tilde r=\frac{7}{2}r-a s_r +\cO(r^2/N^{1/3}).
\end{equation}
Applying \eqref{spc3} with $\nu\to \rho_+$, $N\to \tilde N$ and $r\to \tilde r$ gives
\begin{equation}
 \P\big(Z^{\rho_+}_{o,x^3}< 0\big) \leq e^{-c_0\tilde r^3}.
\end{equation}
Since $s_r\leq 2r$ and $r\leq N^{1/3}/\log(N)$, we have $\tilde r\geq r$ for all $N$ large enough, which proves \eqref{lb1}.

Let $x^4=x^1-15 r N^{2/3} \mathrm{e}_1$. Then by stationarity of the model,
\begin{equation}\label{eq3.48}
 \P\big(Z^{\rho_+}_{o,x^1}>15 rN^{2/3}\big) = \P\big(Z^{\rho_+}_{o,x^4}> 0\big).
\end{equation}
We will apply this time \eqref{spc2}. Denote $\tilde N = N - \frac{15}{2} r N^{2/3}$ and write
\begin{equation}
 x^4=\xi(\rho_+) 2\tilde N-\hat r \tilde N^{2/3} \mathrm{e_3}.
\end{equation}
Solving with respect to $\hat r$ we obtain
\begin{equation}
 \hat r=\frac{7}{2}r-a s_r +\cO(r^2/N^{1/3})\geq r
\end{equation}
for all $N$ large enough. Applying \eqref{spc2} with $\nu\to \rho_+$, $N\to \tilde N$ and $r\to \hat r$ proves \eqref{lb2}.
\end{proof}

\subsection{Uniform sandwiching of geodesics terminating in $\cC^{s_r/2,t_r}$.}
Consider the following assumption.
\begin{assumption}\label{ass}
Let $M_0>0$, $a=3/8$, $s_r\leq \min\{r,4\}$ and make the following assumptions on the parameters:
\begin{equation}\label{eq3.67}
r \leq N^{1/3}/\log(N),\quad M_0\leq M\leq \tfrac1{16} s_r t_r^{-2/3}-4 r t_r^{1/3}.
\end{equation}
\end{assumption}
We shall later discuss this assumption in Remark~\ref{rem:assumptions} below.
Under Assumption~\ref{ass}, the geodesics $\pi^{1/2}_{o,x}$ and $\pi_{y,x}$, for $y\in\cR^{r/2,1/4}$, are controlled by the ones with densities $\rho_+$ and $\rho_-$ for all $x\in \cC^{s_r/2,t_r}$. This is the content of the following result whose proof we defer to the end of this section.
\begin{lemma}\label{cor:og}
Under Assumption~\ref{ass}, there exists $C,c>0$ such that
	\begin{equation}\label{eq3.78}
		\P\Big(\pi^{\rho_-}_{o,x}\preceq \pi^{1/2}_{o,x},\pi_{y,x} \preceq \pi^{\rho_+}_{o,x} \quad \forall x\in \cC^{s_r/2,t_r},y\in \cR^{r/2,1/4}\Big)\geq 1-Ce^{-cM^3}-2e^{-cr^3}
	\end{equation}
for all $N$ large enough.
\end{lemma}

Define
\begin{equation}
\begin{aligned}\label{c}
c^1&=\pi^{\rho_+}_{o,x^1}\cap L_{1-t_r},\\
c^2&=\pi^{\rho_+}_{o,x^2}\cap L_{1-t_r}.
\end{aligned}
\end{equation}
To ease the notation we also denote
\begin{equation}
\begin{aligned}
	w^2&=(1-t_r)N\mathrm{e}_4- as_rN^{2/3}\mathrm{e}_3=x^2-t_rN\mathrm{e}_4, \\
	w^1&=(1-t_r)N\mathrm{e}_4+ as_rN^{2/3}\mathrm{e}_3=x^1-t_rN\mathrm{e}_4.
\end{aligned}
\end{equation}
\begin{lemma}\label{lem:cyl}
	There exists $c,N_0,M_0>0$ such that for $t_r N>N_0$, $r\leq N^{1/3}/\log(N)$, $M\geq M_0$
	\begin{align}
		&\P\Big(w^2-M(t_rN)^{2/3}\mathrm{e}_3 \preceq c^2\preceq w^2+(8rt_rN^{2/3}+M(t_rN)^{2/3})\mathrm{e}_3\Big) \geq 1-e^{-cM^3},\label{lb3}\\
		&\P\Big(w^1-M(t_rN)^{2/3}\mathrm{e}_3 \preceq c^1 \preceq w^1+(8rt_rN^{2/3}+M(t_rN)^{2/3})\mathrm{e}_3 \Big) \geq 1-e^{-cM^3}.\label{lb4}
	\end{align}
\end{lemma}
\begin{proof}
Let $p^2$ be the point of intersection of the characteristic $\xi_+$ starting from $x^2$ with the line $L_{1-t_r}$. We have
\begin{equation}
 p^2=w^2+(4r t_r N^{2/3}+\cO(r^3 t_r))\mathrm{e}_3=w^2+4r t_r N^{2/3}(1+o(1))\mathrm{e}_3
\end{equation}
for $r\leq N^{1/3}/\log(N)$ and $N$ large enough, implying
\begin{equation}
	w^2 \preceq p^2\preceq w^2 + 8rt_r N^{2/3}\mathrm{e}_3=:z^2.
\end{equation}

By the order on geodesics $c^2\preceq \pi^{\rho_+}_{o,x^2+4r t_r N^{2/3}\mathrm{e}_3}$ and if $Z^{\rho_+}_{z^2+M(t_r N)^{2/3}\mathrm{e}_3,x^2+4r t_r N^{2/3}\mathrm{e}_3}<0$, then $\pi^{\rho_+}_{o,x^2+4r t_r N^{2/3}\mathrm{e}_3}\preceq z^2+M(t_r N)^{2/3}\mathrm{e}_3$. Thus
\begin{equation}
 \P(c^2\preceq z^2+M(t_r N)^{2/3}\mathrm{e}_3)\geq \P(Z^{\rho_+}_{z^2+M(t_r N)^{2/3}\mathrm{e}_3,x^2+4r t_r N^{2/3}\mathrm{e}_3}<0).
\end{equation}
Using \eqref{spc2}, the latter is bounded from above by $1-e^{-c_0 M^3}$ provided $M\geq M_0$ and $t_r N\geq N_0$.

A similar bound can be obtained for
\begin{equation}
 \P(w^2-M(t_r N)^{2/3}\mathrm{e}_3\preceq c^2)
\end{equation}
using \eqref{spc3}. Thus we have shown that \eqref{lb3} holds. The proof of \eqref{lb4} is almost identical and thus we do not repeat the details.
\end{proof}

Set
\begin{equation}
\begin{aligned}\label{q}
	q^1&=(1-t_r)N\mathrm{e}_4+ N^{2/3}(a s_r-2 M t_r^{2/3})\mathrm{e}_3,\\
	q^2&=(1-t_r)N\mathrm{e}_4+N^{2/3}(8 r t_r-a s_r+2M t_r^{2/3})\mathrm{e}_3
\end{aligned}
\end{equation}
and define the lines (see Figure~\ref{fig:dim})
\begin{equation}
\begin{aligned}
	D^1&=\{q^1+\alpha t_r N\mathrm{e}_4:0<\alpha<1\},\\
	D^2&=\{q^2+\alpha t_r N\mathrm{e}_4:0<\alpha<1\}.
\end{aligned}
\end{equation}

\begin{lemma}\label{lem:cyl2}
	 There exist $N_1,c,C>0$ such that for every $N\geq N_1$ and $M\leq N^{1/3}/\log(N)$, $r\leq N^{1/3}/\log(N)$,
	\begin{align}
	&\P\Big(D^1 \preceq \pi^{\rho_+}_{o,x^1}\Big)\geq 1- Ce^{-c M^3},\label{ub7}\\
	&\P\Big(\pi^{\rho_+}_{o,x^2} \preceq D^2\Big)\geq 1- Ce^{-c M^3}.\label{ub8}
	\end{align}
\end{lemma}
\begin{proof}
	We will show \eqref{ub8} as \eqref{ub7} can be proven similarly. Let
	\begin{equation}
	u^2=q^2-M(t_rN)^{2/3}\mathrm{e}_3.
	\end{equation}
	By Theorem~\ref{prop1} we have
	\begin{equation}\label{or}
		\P\big(\pi_{u^2,x^2}\preceq D^2\big)\geq 1- Ce^{-c M^3}.
	\end{equation}
	Recall the definition \eqref{c} of $c^1$. By Lemma~\ref{lem:cyl}
	\begin{equation}
	\P\big(c^2\preceq u^2\big)\geq 1- e^{-c M^3},
	\end{equation}
	which implies that
	\begin{equation}\label{or2}
			\P\big(\pi^{\rho_+}_{o,x^2}\preceq \pi_{u^2,x^2}\big)\geq 1- Ce^{-c M^3}.
	\end{equation}
	\eqref{or} and \eqref{or2} imply \eqref{ub8}.
\end{proof}

\begin{remark}\label{rem:assumptions}
Now we can discuss the origin of the conditions in Assumption~\ref{ass}.
The bound on $r$ comes from Lemma~\ref{lem:cyl2}. The condition on $M$ is a consequence of the conditions $q^2\preceq (1-t_r)N\mathrm{e}_4-\tfrac14 s_r N^{2/3} \mathrm{e_3}$ and also $(1-t_r)N\mathrm{e}_4+\tfrac14 s_r N^{2/3} \mathrm{e}_3\preceq q^1$. As we want  $M$ to grow to infinity we need to take $t_r \ll s_r / r$.
\end{remark}

For $0<\tau<1$ and $\sigma\in\R_+$, define the anti-diagonal segment
\begin{equation}
\cI^{\sigma,\tau}=\{(1-\tau) N\mathrm{e}_4+ i\mathrm{e}_3, i\in [-\tfrac\sigma{2} N^{2/3},\tfrac{\sigma}{2}N^{2/3}]\},
\end{equation}
located right below the cylinder $\cC^{\sigma,\tau}$.
Define the events
\begin{equation}
\begin{aligned}
	\cO&=\big\{Z^{\rho_-}_{o,x}\in [-15rN^{2/3},-rN^{2/3}],Z^{\rho_+}_{o,x}\in [rN^{2/3},15rN^{2/3}]\quad \forall x\in \cC^{s_r/2,t_r}\big\},\\
	\cB&= \Big\{\{\pi^{\rho_-}_{o,x} \cap \cI^{s_r,t_r}\neq \emptyset\} \cap \{\pi^{\rho_+}_{o,x} \cap \cI^{s_r,t_r}\neq \emptyset\} \quad \forall x\in \cC^{s_r/2,t_r}\Big \}.
\end{aligned}
\end{equation}

\begin{corollary}\label{cor:up}
Under Assumption~\ref{ass} there exists $C,c,N_0>0$ such that for $N>N_0$
\begin{align}
\P(\cO)&\geq 1-2e^{-cr^3}-C e^{-cM^3},\label{r1}\\
\P(\cB)&\geq 1-e^{-cM^3}\label{r2}.
\end{align}
\end{corollary}
\begin{proof}
	It might be helpful to take a look at Figure~\ref{fig:dim} while reading the proof. We prove in details the statements for $\rho_+$, since the proof for $\rho_-$ is almost identical.

By our choice of parameters,
\begin{equation}
	D^2\preceq \cC^{s_r/2,t_r} \preceq D^1.
\end{equation}
By Lemma~\ref{lem:cyl2} and order of geodesics, with probability at least $1-C e^{-c M^3}$,
\begin{equation}\label{eq3.75}
\pi^{\rho_+}_{o,x^2}\preceq \pi^{\rho_+}_{o,x}\preceq \pi^{\rho_+}_{o,x^1}
\end{equation}
for all $x\in \cC^{s_r/2,t_r}$. By Lemma~\ref{lem:sb} the exit point of the geodesics to $x^1$ and $x^2$ for the stationary model with density $\rho_+$ lies between $rN^{2/3}$ and $15r N^{2/3}$ with probability at least $1-e^{-c r^3}$, which leads to \eqref{r1}.

To prove \eqref{r2}, first notice that Assumption ~\ref{ass} implies
\begin{equation}
w^1+(8rt_rN^{2/3}+M(t_rN)^{2/3})\mathrm{e}_3 \preceq w^1+\tfrac12 s_r\mathrm{e}_3
\end{equation}
and
\begin{equation}
w^2-\tfrac12 s_r \mathrm{e}_3 \preceq w^2-M(t_rN)^{2/3}\mathrm{e}_3.
\end{equation}
Thus by Lemma~\ref{lem:cyl} we know that the crossing of $\pi^{\rho_+}_{o,x^1}$ and $\pi^{\rho_+}_{o,x^2}$ with $\cI^{s_r,t_r}$ occurs with probability at least $1-e^{-c M^3}$. This, together with \eqref{eq3.75} implies \eqref{r2}.
\end{proof}

\begin{proof}[Proof of Lemma~\ref{cor:og}]
	Consider the straight line going from $(0,rN^{2/3})$ to $x^1$ parameterized by $(u,l_1(u))$ with
\begin{equation}
l_1(u)=r N^{2/3}+u \frac{N-(\tfrac38 s_r +r)N^{2/3}}{N+\tfrac38 s_r N^{2/3}},
\end{equation}
and the straight line $(u,l_2(u))$, which overlaps in its first part with the boundary of $\cR^{r/2,1/4}$, defined through
\begin{equation}
l_2(u)=\frac{r}{2}N^{2/3}+u.
\end{equation}
By our assumption $s_r\leq r$,
\begin{equation}\label{lb}
		\inf_{0\leq u \leq \frac{N}{4}}( l_1(u)-l_2(u))\geq \frac{r}{2}N^{2/3}-\frac{\frac{7}{4}r N^{2/3}}{N+\frac38 r N^{2/3}} \frac{N}{4}\geq \frac1{16}r N^{2/3}.
\end{equation}
	It follows from \eqref{lb} and Theorem~\ref{prop1} that for some $C,c>0$
\begin{equation}\label{eq4.80}
		\P(\Gamma_u^l(\pi_{(0,rN^{2/3}),x^1})<l_2(u) \textrm{ for some }0\leq u \leq \tfrac{N}{4})\leq Ce^{-cr^3}.
\end{equation}
By the analogue of Lemma~\ref{lem:cyl2} for $\rho_-$, with probability at least $1-C_1 e^{-c_1 M^3}$, $\pi^{\rho_-}_{o,x}\preceq \pi^{\rho_-}_{o,x^1}$ for all $x\in \cC^{s_r/2,t_r}$. Furthermore, on the event $\cA_2$, $\rho^{\rho_-}_{o,x^1}$ starts from a point above $(0,r N^{2/3})$. Combining these two facts with \eqref{eq4.80} we get
	\begin{equation}\label{lb5}
		\P\Big(\pi^{\rho_-}_{o,x}\preceq \pi_{(0,rN^{2/3}),x^1} \preceq \cR^{r/2,1/4} \quad \forall x\in \cC^{s_r/2,t_r}\Big)\geq 1-C_1 e^{-c_1 M^3}-Ce^{-cr^3}.
	\end{equation}
	A similar result can be obtained for $\pi^{\rho_+}_{o,x}$, which combined with \eqref{lb5} gives
	\begin{equation}\label{eq4.82}
		\P\Big(\pi^{\rho_-}_{o,x} \preceq \cR^{r/2,1/4} \preceq \pi^{\rho_+}_{o,x}\quad \forall x\in \cC^{s_r/2,t_r}\Big)\geq 1-2C_1 e^{-c_1 M^3}-2Ce^{-cr^3}.
	\end{equation}
By the order of geodesics, on the event of \eqref{eq4.82}, every geodesic starting in $\cR^{r/2,1/4}$ and ending at $x$, is sandwiched between $\pi^{\rho_-}_{o,x}$ and $\pi^{\rho_+}_{o,x}$.

Next, for each point $x\in \cC^{s_r/2,t_r}$, its associated density  $\rho(x)$ satisfies $|\rho(x)-\tfrac12|\leq N^{-1/3}$ for all $s_r\leq 4(1-t_r)$ and $N$ large enough. By \eqref{spc1} of Lemma~\ref{spc}, with probability at least $1-e^{-c_0 r^3}$ the exit point of $\pi^{1/2}_{o,x}$ is also between $-r N^{2/3}$ and $r N^{2/3}$. Thus by appropriate choice of constants $C,c$,  the sandwitching of $\pi^{1/2}_{o,x}$ in \eqref{eq3.78} holds.
\end{proof}

\section{Lower bound for the probability of no coalescence }\label{sectLowerBound}

\subsection{Point-to-point case: proof of Theorem~\ref{thm:coal}}
Using the results of Section~\ref{SectLocalStat} we first relate the bound for the coalescing point of $\pi^{1/2}_{o,x}$ and $\pi_{o,x}$ to that of the coalescing point of $\pi^{\rho_+}_{o,x}$ and $\pi^{\rho_-}_{o,x}$.
\begin{lemma}\label{cor:cub}
Under Assumption~\ref{ass}, there exists $C,c>0$ such that
\begin{equation}\label{eq4.1}
\begin{aligned}
	&\P\Big(C_p(\pi^{1/2}_{o,x},\pi_{y,x})\leq L_{1-t_r} \quad \forall x\in \cC^{s_r/2,t_r},y\in\cR^{r/2,1/4}\Big)\geq 1-Ce^{-cM^3}-2e^{-cr^3}\\
&-\P\Big(\exists x\in \cC^{s_r,t_r}: C_p(\pi^{\rho_+}_{o,x},\pi^{\rho_-}_{o,x})\geq \cI^{s_r/2,t_r},\cB \Big)
\end{aligned}
\end{equation}
for all $N$ large enough.
\end{lemma}
\begin{proof}
We bound the probability of the complement event. Define the event
\begin{equation}\label{eventG}
\cG=\{\pi^{\rho_-}_{y,x}\preceq \pi^{1/2}_{o,x},\pi_{y,x} \preceq \pi^{\rho_+}_{y,x} \quad \forall x\in \cC^{s_r/2,t_r}, y\in\cR^{r/2,1/4}\}.
\end{equation}
Then
\begin{equation}
\begin{aligned}
& \P\Big(\exists x \in \cC^{s_r/2,t_r},y\in\cR^{r/2,1/4}: C_p(\pi^{1/2}_{o,x},\pi_{y,x})>L_{1-t_r}\Big) \leq \P(\cB^c)+\P(\cG^c)\\
&+ \P\Big(\{\exists x \in \cC^{s_r/2,t_r},y\in\cR^{r/2,1/4}: C_p(\pi^{1/2}_{o,x},\pi_{y,x})>L_{1-t_r}\}\cap \cB\cap \cG\Big).
\end{aligned}
\end{equation}
Note that if $\cB$ and $\cG$ hold, then both geodesics $\pi^{1/2}_{o,x}$ and $\pi_{y,x}$ are sandwiched between $\pi^{\rho_-}_{o,x}$ and $\pi^{\rho_+}_{o,x}$ and their crossings with the line $L_{1-t_r}$ occurs in the segment $\cI^{s_r,t_r}$. Furthermore, if $C_p(\pi^{1/2}_{o,x},\pi_{y,x})>L_{1-t_r}$ and $\cG$ holds, then also $C_p(\pi^{\rho_+}_{o,x},\pi^{\rho_-}_{o,x})>L_{1-t_r}$. This, together with Corollary~\ref{cor:up} and Lemma~\ref{cor:og} proves the claim.
\end{proof}
Thus, to prove Theorem~\ref{thm:coal} it remains to get an upper bound for the last probability in \eqref{eq4.1}.

	For $x,y,z\in\Z^2$ such that $x\leq y \leq z$, let $\gamma_{x,z}$ be an up-right path going from $x$ to $z$. Define the exit point of $\gamma_{x,z}$ with respect to the point $y$
	\begin{equation}
		\cZ_y(\gamma_{x,z})=\sup\{u\in\gamma_{x,z}:u_1=y_1 \text{ or } u_2=y_2\}.
	\end{equation}
	Define the sets
\begin{equation}
\begin{aligned}
\cH&=\{(1-t_r)N\mathrm{e}_4+\tfrac {s_r}2N^{2/3}\mathrm{e}_3-i\mathrm{e}_1, \quad 1 \leq i \leq \tfrac {s_r}2N^{2/3}\},\\
\cV&=\{(1-t_r)N\mathrm{e}_4-\tfrac {s_r}2N^{2/3}\mathrm{e}_3-i\mathrm{e}_2, \quad 1 \leq i \leq \tfrac {s_r}2N^{2/3}\},
\end{aligned}
\end{equation}
and the point
\begin{equation}
v^c=[(1-t_r)N-\tfrac{s_r}{2} N^{2/3}]\mathrm{e}_4. 
\end{equation}
	 Define the event
	\begin{equation}
		E_1=\{\cZ_{v^c}(\pi^{\rho_+}_{o,x})\in \cH\cup\cV ,\cZ_{v^c}(\pi^{\rho_-}_{o,x})\in \cH\cup\cV \quad \forall x\in \cC^{s_r/2,t_r}\}.
	\end{equation}
	Note that (see Figure~\ref{fig:coal})
	\begin{equation}
		E_1=\cB,
	\end{equation}
	since to cross the set $\cI^{s_r,t_r}$ the geodesic must cross either $\cH$ or $\cV$ and, viceversa, if the geodesic crosses $\cH\cup\cV$, then it crosses also $\cI^{s_r,t_r}$.
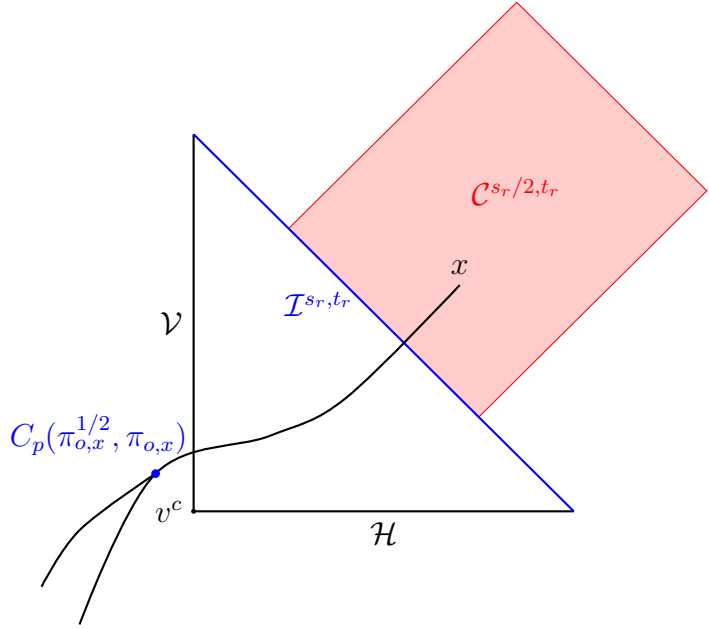
\begin{figure}
\begin{center}
	\begin{tikzpicture}[scale=0.5, every node/.style={transform shape}]
	node/.style={transform shape}]
	\def\s{2.0}
	\draw [black,thick](0,0) -- (0,10);
	\node [scale=\s,left,black] at (0,5) {$\cV$};
	\draw [black,thick](0,0) -- (10,0);
	\node [scale=\s,below,black] at (5,0) {$\cH$};
	\draw [color=red,fill=red!20] (2.5,7.5) -- (7.5,2.5) -- (13.5,8.5) --(8.5,13.5) -- (2.5,7.5);
	\draw [blue,thick] (0,10)--(10,0);
	\node [scale=\s,red] at (8.5,8.5) {$\cC^{s_r/2,t_r}$};
	\node [scale=\s,blue] at (3.3,5.5) {$\cI^{s_r,t_r}$};
	\draw[black,thick] plot [smooth,tension=0.6] coordinates{(-3,-3) (-1,1) (2,2) (4,3) (7,6) };
	\draw[black,thick] plot [smooth,tension=0.6] coordinates{(-4,-2) (-3,-0.5) (-1,1)};
	\node [scale=\s,above,blue] at (-2.5,1.2) {$C_p(\pi^{1/2}_{o,x},\pi_{o,x})$};
	\node [scale=\s,above] at (7,6) {$x$};
	\draw [blue,fill](-1,1) circle (1mm);
	\draw [black,fill](0,0) circle (0.5mm);
	\node [scale=\s,left] at (0,0) {$v^c$};
	\end{tikzpicture}
	\caption{On the event $E_1\cap E_2$, the geodesics $\pi^{\rho_+}_{o,x}$ and $\pi^{\rho_-}_{o,x}$ coalesce before crossing $\cI^{s_r,t_r}$.}
\label{fig:coal}
\end{center}
\end{figure}

On the edges of $\Z_{\geq 0}^2$ define the random field $B$ through
	\begin{equation}\label{bf3}
		B_{x,x+\mathrm{e}_k}=G(\pi_{o,x+\mathrm{e}_k})-G(\pi_{o,x}),\quad k=1,2
 \end{equation}
for all $x>o$. Similarly we define
\begin{equation}
	\begin{aligned} B^{\rho_+}_{x,x+\mathrm{e}_k}&=G(\pi^{\rho_+}_{o,x+\mathrm{e}_k})-G(\pi^{\rho_+}_{o,x}), \quad k=1,2,\\ B^{\rho_-}_{x,x+\mathrm{e}_k}&=G(\pi^{\rho_-}_{o,x+\mathrm{e}_k})-G(\pi^{\rho_-}_{o,x}),\quad k=1,2.\label{bf2}
	\end{aligned}
\end{equation}
	 One can couple the random fields $B,B^{\rho_-}$ and $B^{\rho_+}$ (see \cite[Theorem 2.1]{FS18}) such that
\begin{equation}
	 \begin{aligned}\label{og}
	 	B^{\rho_-}_{x-\mathrm{e}_1,x}&\leq B_{x-\mathrm{e}_1,x} \leq B^{\rho_+}_{x-\mathrm{e}_1,x},\\
	 	B^{\rho_+}_{x-\mathrm{e}_2,x}&\leq B_{x-\mathrm{e}_2,x} \leq B^{\rho_-}_{x-\mathrm{e}_2,x}.
	 \end{aligned}
\end{equation}

Let $o\leq v\leq x$. Then, since each geodesic has to pass either by one site on the right or above $v$, we have
\begin{equation}
\begin{aligned}
	 G_{o,x}=G_{o,v}+\max\Big\{&\sup_{0\leq l \leq x_1-v_1}\sum_{i=0}^l B_{v+i\mathrm{e}_1,v+(i+1)\mathrm{e}_1}+G_{v+(l+1)\mathrm{e}_1+\mathrm{e}_2,x},\\
&\sup_{0\leq l\leq x_2-v_2}\sum_{i=0}^l B_{v+i\mathrm{e}_2,v+(i+1)\mathrm{e}_2}+G_{v+(l+1)\mathrm{e}_2+\mathrm{e}_1,x}\Big\}.
\end{aligned}
\end{equation}
Thus setting $v=v^c$, on the event $E_1\cap \cG$, for every $x\in\cC^{s_r/2,t_r}$
\begin{equation}
\begin{aligned}\label{wr}
G_{o,x}=G_{o,v}+\max\Big\{&\sup_{u\in \cH}\sum_{i=0}^{u_1-v^c_1} B_{v^c+i\mathrm{e}_1,v^c+(i+1)\mathrm{e}_1}+G_{u+\mathrm{e}_2,x},\\
&\sup_{u\in \cV}\sum_{i=0}^{u_2-v^c_2} B_{v^c+i\mathrm{e}_2,v^c+(i+1)\mathrm{e}_2}+G_{u+\mathrm{e}_1,x}\Big\}.
\end{aligned}
\end{equation}
$G^{\rho_+}_{o,x}$ and $G^{\rho_-}_{o,x}$ can be decomposed in the same way.

This shows that on the event $E_1\cap \cG$ the restriction of the geodesics $\pi_{o,x},\pi^{\rho_+}_{o,x},\pi^{\rho_-}_{o,x}$ to $\Z^2_{>v^c}$ is a function of the weights \eqref{bf3}--\eqref{bf2} and the bulk weights east-north to $v^c$. More precisely, define $\cE^B_{v^c}$ to be the set of edges in the south-west boundary of $\Z^2_{>v^c}$ and that are incident to $\cH\cup \cV$ i.e.
\begin{equation}
	 \begin{aligned}
	 	\cE_{v^c}^H&=\{(x-\mathrm{e}_2,x):x\in \cH\},\\
	 	\cE_{v^c}^V&=\{(x-\mathrm{e}_1,x):x\in \cV\},\\
	 	\cE^B_{v^c}&=\cE_{v^c}^H\cup \cE_{v^c}^V.
	 \end{aligned}
\end{equation}
The representation \eqref{wr} show that on the event $E_1\cap \cG$, for every $x\in\cC^{s_r/2,t_r}$, the restrictions of the geodesics $\pi_{o,x},\pi^{\rho_+}_{o,x},\pi^{\rho_-}_{o,x}$ to $\Z_{> v^c}$ are functions of the bulk weights
	 \begin{equation}\label{w}
	 	\{\omega_x\}_{x\in \Z^2_{>v^c}}
	 \end{equation}
	 and the boundary weights
	 \begin{equation}\label{bf5}
	 	\{B_e\}_{e\in\cE^B_{v^c}}, \{B^{\rho_+}_e\}_{e\in\cE^B_{v^c}}\textrm{ and }\{B^{\rho_-}_e\}_{e\in\cE^B_{v^c}}
	 \end{equation}
respectively. The stationary geodesics with densities $\rho_+$ and $\rho_-$ will coalesce before reaching $\cI^{s_r,t_r}$ if the following event holds true:
	 \begin{equation}
	 	E_2=\{\exists e\in\cE^B_{v^c}: B^{\rho_+}_e= B^{\rho_-}_e \}.
	 \end{equation}
	 \begin{lemma}\label{lem:e1e2}
We have
	 	\begin{equation}
	 		 \{\exists x\in \cC^{s_r/2,t_r}: C_p(\pi^{\rho_+}_{o,x},\pi^{\rho_-}_{o,x})\geq \cI^{s_r,t_r}, \cB \} \subseteq E_1\cap E_2.
	 	\end{equation}
	 \end{lemma}
 \begin{proof}
The representation \eqref{wr} for the stationary models imply that $G^{\rho_-}_{o,x}$ and $G^{\rho_+}_{o,x}$ are functions of the weights in \eqref{w} and the stationary weights in \eqref{bf5}. This implies that on $(E_1\cap E_2)^c$
\begin{equation}
	C_p(\pi^{\rho_+}_{o,x},\pi^{\rho_-}_{o,x})\leq \cH\cup \cV \leq \cI^{s_r,t_r} \quad \forall x\in \cC^{s_r/2,t_r},
\end{equation}
which implies the result.
 \end{proof}

Thus it remains to find an upper bound for $\P(E_2)$. For $m>0$,  let us define
\begin{equation}
	\cA^m=\big\{B^{\rho_+}_{v^c+i\mathrm{e}_1,v^c+(i+1)\mathrm{e}_1}=B^{\rho_-}_{v^c+i\mathrm{e}_1,v^c+(i+1)\mathrm{e}_1}, 0\leq i \leq m\big\}. 
\end{equation}
Recall that $\rho_+=1/2+rN^{-1/3}$. In~\cite{BBS20} the following bound is proven.
\begin{lemma}[Lemma~5.9 of~\cite{BBS20}]
	Let $m\geq1$. For $0<\theta<\rho_+$, it holds
\begin{equation}
	\begin{aligned}\label{tub}
	\P(\cA^m)&\geq 1- \frac{2rN^{-1/3}}{\frac12+rN^{-1/3}}\\
&+\frac{\frac12-rN^{-1/3}}{\frac12+rN^{-1/3}}\Bigg[1+\frac{2r\theta N^{-1/3}+\theta^2}{\frac14-(r^2N^{-2/3}+2rN^{-1/3}\theta+\theta^2)}\Bigg]^{m}\frac{1}{1+2\theta r^{-1}N^{1/3}}.
	\end{aligned}
\end{equation}
\end{lemma}
\begin{corollary}\label{cor:co}
	There exists $C>0$, such that for every $r>0$ and $0<\eta<1/4000$, 
	\begin{equation}\label{aa}
		\P(\cA^{\eta r^{-2}N^{2/3}})\geq 1-C\eta^{1/2}.
	\end{equation}
for all $N$ large enough.
\end{corollary}
	\begin{proof}
		We set $\theta=\eta^{-1/2}rN^{-1/3}$ and plug this into \eqref{tub}. Taking $N\to\infty$ we obtain
		\begin{equation}
			\lim_{N\to\infty}\P(\cA^{\eta r^{-2}N^{2/3}})\geq 1- \frac{e^{4+8\sqrt{\eta}}\sqrt{\eta}}{1+2\sqrt{\eta}}.
		\end{equation}
    Taking for instance $C=62$, then for all $0<\eta<1/4000$, we have $1\geq C \sqrt{\eta}\geq \frac{e^{4+8\sqrt{\eta}}\sqrt{\eta}}{1+2\sqrt{\eta}}$, which implies the result. 
	\end{proof}
We are now able to bound the probability of $E_2$.
\begin{corollary}\label{cor:e2}
		There exists $C>0$, such that for every $r>0$ and $0<s_r<1/(2000 r^2)$ 
	\begin{equation}
	\P(E_2)\leq Cs_r^{1/2} r
	\end{equation}
for all $N$ large enough.
\end{corollary}
\begin{proof}
	Define
\begin{equation}
	\begin{aligned}
		E^H_2&=\{\exists e\in\cE_{v^c}^H : B^{\rho_+}_e\neq B^{\rho_-}_e \},\\
		E^V_2&=\{\exists e\in\cE_{v^c}^V : B^{\rho_+}_e\neq B^{\rho_-}_e \},
	\end{aligned}
\end{equation}
	and note that
	\begin{equation}
		E_2=E^H_2\cup E^V_2.
	\end{equation}
	By the symmetry of the problem, it is enough to show
	\begin{equation}
			\P(E^H_2)\leq Cs_r^{1/2} r.
	\end{equation}
	Apply Corollary \ref{cor:co} with  $m=\frac{s_r}{2}N^{2/3}$ i.e.\  with $\eta=\frac{s_r r^2}{2}$. Then $\P\big((E^H_2)^c\big)=\P(\cA^{\eta r^{-2}N^{2/3}})$ and \eqref{aa} gives the claimed result.
\end{proof}

\begin{corollary}
Consider the parameters satisfying Assumption~\ref{ass} and $s_r r^2<1/4000$.  Then, there exist constants $c,C>0$ such that
\begin{equation}\label{eq5.28}
		\P\Big(C_p(\pi^{1/2}_{o,x},\pi_{y,x})\leq L_{1-t_r} \quad \forall x\in \cC^{s_r/2,t_r},y\in\cR^{r/2,1/4}\Big) \geq 1-Ce^{-cM^3}-e^{-cr^3}-Cs_r^{1/2}r
	\end{equation}
for all $N$ large enough.
\end{corollary}
\begin{proof}
	By Lemma~\ref{lem:e1e2}
	\begin{equation}
		\P\Big(\exists x\in \cC^{s_r/2,t_r}: C_p(\pi^{\rho_+}_{o,x},\pi^{\rho_-}_{o,x})\geq \cI^{s_r,t_r}, \cB \Big) \leq \P(E_2).
	\end{equation}
	The result follows from Corollary~\ref{cor:e2} and Lemma~\ref{cor:cub}.
\end{proof}
\begin{proof}[Proof of Theorem~\ref{thm:coal}]
	To end the proof of Theorem~\ref{thm:coal}, we just need to express the parameters $r,s_r,t_r,M$ in terms of $\delta$ so that Assumption~\ref{ass} is in force. For a small $\delta>0$ consider the scaling
	\begin{equation}\label{eq4.31}
	\begin{aligned}
	s_r&=2\delta,\\
	t_r&=\delta^{3/2}/(\log(1/\delta))^3,\\
	r&=\tfrac14\log(1/\delta),\\
	M&=\tfrac14\log(1/\delta).
	\end{aligned}
	\end{equation}
	Let us verify the assumptions. For all $\delta\leq 0.05$, the last inequality in \eqref{eq3.67} holds true and also $s_r<r$. Finally, for $\delta\leq \exp(-4 M_0)$, we have $M\geq M_0$ as well. For small $\delta$, the largest error term in \eqref{eq5.28} is $C s_r^{1/2} r$, which however goes to $0$ as $\delta$ goes to $0$. 
\end{proof}
\begin{remark}
Of course, \eqref{eq4.31} is not the only possible choice of parameters. For instance, one can take $t_r=\delta^{3/2}/\log(1/\delta)$, but then we have to take smaller values of $r$ and $M$, e.g., $r=M=\tfrac{1}{16}\sqrt{\log(1/\delta)}$, for which the last inequality in \eqref{eq3.67} holds for $\delta\leq 0.01$, but the decay of the estimate $e^{-c r^3}$ and $e^{-c M^3}$ are much slower.
\end{remark}

\subsection{General initial conditions: proof of Theorem~\ref{thm:coalGeneralIC}}
In this section we prove Theorem~\ref{thm:coalGeneralIC}. The strategy of the proof is identical to the one of Theorem~\ref{thm:coal} and thus we will focus only on the differences.
\begin{proof}[Proof of Theorem~\ref{thm:coalGeneralIC}]
	Here we think of the stationary LPP as leaving from the line $\cL=\{x\in\Z^2| x_1+x_2=0\}$ rather than the point $o=(0,0)$. The first ingredient of the proof is a version of Lemma~\ref{lem:sb}, where we adapt Lemma~\ref{lem:sb} so that the statement holds with  $Z^{\rho_\pm}_{\cL,x^k}$    rather than $Z^{\rho_\pm}_{o,x^k}$ \footnote{Since the geodesics have slope very close to $1$, one might expect that $r N^{2/3}$ and $15 r N^{2/3}$ should be replaced by their half. However the statement don't need to be changed since we did not choose these boundary to be sharp.}. Next we use line-to-point versions of \eqref{eq3.44} and \eqref{eq3.48} which indeed can be obtained as  $\{Z^{\rho_+}_{o,x^3}<0\}=\{Z^{\rho_+}_{\cL,x^3}<0\}$ and therefore the bounds from Lemma~\ref{spc} can still be applied.
	
	Lastly, for general initial conditions, the exit point of $Z^{h_0}_{\cL,x}$ is not $0$ anymore. By Assumption~\ref{ass_IC} it is  between $-r N^{2/3}$ and $r N^{2/3}$ with $r=\log(\delta^{-1})$ with probability at least $1-Q(\delta)$. Thus we replace the bound $1-e^{-c r^3}$ for $r=\log(\delta^{-1})$ with $1-Q(\delta)$. The rest of the proof is unchanged.
\end{proof}
\section{Upper bound for the probability of no coalescence}\label{SectUpperBound}
In this section we will prove Theorem~\ref{thm:LBcoal}, but for this we need some preparations.

\subsection{Coupling of stationary models with distinct densities}
Let $\arrv=(\arr_j)_{j\in\Z}$ and $\servv=(\serv_j)_{j\in\Z}$ be two independent sequences of i.i.d.\ exponential random variables of intensity $\beta$ and $\alpha$ respectively, where $0<\beta<\alpha<1$. We think of $a_j$ as the inter-arrival time between customer $j$ and customer $j-1$, and of $s_j$ as the service time of customer $j$. The waiting time of the $j$'th customer is given by
\begin{equation}\label{waittime}
w_j=\sup_{i\leq j}\Big(\sum_{k=i}^{j}s_{k-1}-a_k\Big)^+.
\end{equation}
The distribution of $w_0$ (and by stationarity the distribution of any $w_j$ for $j\in \Z$) is given by
\begin{equation}\label{des}
\P(w_0\in dw):=f(dw)=\big(1-\frac{\beta}{\alpha}\big)\delta_0(dw)+\frac{(\alpha-\beta)\beta}{\alpha}e^{-(\alpha-\beta)w}dw.
\end{equation}
The queueing map $D:\R_+^\Z\times \R_+^\Z\rightarrow\R_+^\Z$ takes the sequence of interarrival times and the service times and maps them to the inter-departure times
\begin{equation}
\begin{aligned}
\depav&=D(\arrv,\servv),\\
d_j&=(w_{j-1}+s_{j-1}-a_j)^-+s_j.
\end{aligned}
\end{equation}
We denote by $\nu^{\beta,\alpha}$ the distribution of $(D(\arrv,\servv),\servv)$ on $\R_+^\Z\times \R_+^\Z$, that is
\begin{equation}\label{qd}
\nu^{\beta,\alpha}\sim (D(\arrv,\servv),\servv).
\end{equation}
By Burke's Theorem~\cite{Bur56} $D(\arrv,\servv)$ is a sequence of i.i.d.\ exponential random variables of intensity $\beta$, consequently, the measure $	\nu^{\beta,\alpha}$ is referred to as a stationary measure of the queue.
One can write
\begin{equation}\label{eqEjs}
d_j=e_j+s_{j},
\end{equation}
where $e_j$ is called the $j$'th \textit{idle time} and is given by
\begin{equation}\label{e}
e_j=(w_{j-1}+s_{j-1}-a_j)^-.
\end{equation}
$e_j$ is the time between the departure of customer $j-1$ and the arrival of customer $j$ in which the sever is idle. Define
\begin{align}
x_j=s_{j-1}-a_j,
\end{align}
and the summation operator
\begin{equation}\label{S}
S^k_l=\sum_{i=k}^lx_i.
\end{equation}
Summing $e_j$ we obtain the cumulative idle time (see Chapter 9.2, Eq. 2.7 of~\cite{WhittBook}).
\begin{lemma}[Lemma A1 of~\cite{BBS20}]
	For any $k\leq l$
	\begin{equation}\label{sume}
	\sum_{i=k}^l e_i=\Big(\inf_{k\leq i\leq l}w_{k-1}+S^k_i\Big)^-.
	\end{equation}
\end{lemma}
It has long been known that the LPP on the lattice can be seen as queues in tandem. In particular, the stationary distribution for LPP can be seen as a stationary distribution of queues in tandem. In \cite{FS18} Fan an Sepp\"al\"ainen found the multi-species stationary distribution for LPP. For $0<\beta<\alpha<1$, let $I^\beta=\{I^\beta_i\}_{i\in\Z}$ and $I^\alpha=\{I^\alpha_i\}_{i\in\Z}$ be two i.i.d.\ random sequences such that
\begin{equation}\label{wg}
I^\beta_1\sim \text{Exp}(1-\beta)\textrm{ and }
I^\alpha_1\sim \text{Exp}(1-\alpha).
\end{equation}

Let $x\in \Z$ such that $o=(0,0)\leq N\mathrm{e}_4+x \mathrm{e}_1$. Let $G^\alpha,G^\beta$ be stationary LPP as in Section~\ref{SectStatLPP} with the weights in \eqref{wg}. Define the sequences $I^{\beta,x}$ and $I^{\alpha,x}$ by
\begin{equation}
\begin{aligned}
I^{\beta,x}_i&=G^\beta_{(1-t_r)N\mathrm{e}_4+i\mathrm{e}_1}-G^\beta_{(1-t_r)N\mathrm{e}_4+(i-1)\mathrm{e}_1} \quad\textrm{for }i>x,\\
I^{\alpha,x}_i&=G^\alpha_{(1-t_r)N\mathrm{e}_4+i\mathrm{e}_1}-G^\alpha_{(1-t_r)N\mathrm{e}_4+(i-1)\mathrm{e}_1} \quad\textrm{for }i>x.
\end{aligned}
\end{equation}
The multi-species results in \cite{FS18}, in particular Theorem~2.3 of~\cite{FS18}, show that if we take $(I^{\alpha},I^{\beta})\sim \nu^{1-\alpha,1-\beta}$ then
\begin{align}
(I^{\alpha,x},I^{\beta,x})\sim \nu^{1-\alpha,1-\beta}|_{x+\R^{\Z_+}}.
\end{align}
where $\nu^{1-\alpha,1-\beta}|_{x+\R^{\Z_+}}$ is the restriction of $\nu^{\beta,\alpha}$ to $x+\R^{\Z_+}$.

\subsection{Control on the stationary geodesics at the $(1-t_r)N$.}
As main ingredient in the proof, in Proposition \ref{prop:Lb} below, we show that with positive probability the geodesics ending at $N\mathrm{e}_4$ for the stationary models with density $\rho_+$ and $\rho_-$ do not coalesce before time $(1-t_r)N$.

Let $r_0>0$ to be determined later, $z_0=-r_0t_r^{2/3} N^{2/3}$, $z_1=r_0t_r^{2/3} N^{2/3}$, and define
\begin{equation}\label{ex}
\cH^{\rho}=\sup\{i\in\Z |(1-t_r)N\mathrm{e}_4+(i,0)\in \pi^\rho_{o,N\mathrm{e}_4}\}
\end{equation}
be the exit point of the geodesic $\pi^\rho_{o,N\mathrm{e}_4}$ with respect to the horizontal line $(\Z,(1-t_r)N)$, which geometrically is at position $\widetilde \cH^\rho=(1-t_r)N\mathrm{e}_4+(\cH^\rho,0)$.
Define
\begin{equation}
I= I_-\cup I_+\textrm{ where } I_-=\{z_0,\ldots,0\}, I_+=\{1,\ldots,z_1\}.
\end{equation}

\begin{proposition}\label{prop:Lb}
	Under the choice of parameters in \eqref{eq4.31}, there exist $C,\delta_0>0$ such that for $\delta<\delta_0$ and $N$ large enough
	\begin{equation}
	\P(\cH^{{\rho_-}}\in I_-, \cH^{{\rho_+}}>0 )\geq C\delta^{1/2}.
	\end{equation}
\end{proposition}
Before we turn to the proof of Proposition \ref{prop:Lb}, we need some preliminary results. The following lemma shows that with probability close to $1/2$
$\pi^\rho_{o,N\mathrm{e}_4}$ will cross the interval $I$ from its left half.
\begin{lemma}\label{lem:gl}
	Under \eqref{eq4.31}, there exists $\delta_0,c>0$ such that for $\delta<\delta_0$ and for large enough $N$
	\begin{equation}
	\P(\cH^{{\rho_-}}\in I_-)\geq \frac12-e^{-cr_0^3}.
	\end{equation}
\end{lemma}
\begin{proof}
	Let $v=(1-t_r)Ne_4-r_0t_r^{2/3}N^{2/3}e_1$. Note that
	\begin{equation}\label{eq5}
	\{Z^{\rho_-}_{v,Ne_4}\in [0,-z_0]\}=\{\cH^{{\rho_-}}\in I_-\}.
	\end{equation}
	Moreover
	\begin{equation}\label{eq4}
	\P(Z^{\rho_-}_{v,Ne_4}\in [0,-z_0])=\P(Z^{\rho_-}_{v,Ne_4}\leq -z_0)-\P(Z^{\rho_-}_{v,Ne_4}< 0).
	\end{equation}
 The exit point $Z^\rho_{v,Ne_4}$ is stochastically monotone in $\rho$, that is,
	\begin{equation}
	\P(Z^\rho_{v,Ne_4}\leq x)\geq \P(Z^\lambda_{v,Ne_4}\leq x) \quad \text{for $\rho\leq\lambda$}.
	\end{equation}
	It follows that
	\begin{equation}
	\begin{aligned}
		\P(Z^{\rho_-}_{v,Ne_4}\leq -z_0)&\geq \P(Z^{1/2}_{v,Ne_4}\leq -z_0)=\P(Z^{1/2}_{v-z_0,Ne_4}\leq 0)\\
    &=\P(Z^{1/2}_{(1-t_r)Ne_4,Ne_4}\leq 0)=1/2\label{eq},
	\end{aligned}
	\end{equation}
	where the last equality follows from symmetry. Consider the characteristic $\rho_-$ emanating from $Ne_4$ and its intersection point $c^0$ with the set $\{(1-t_r)Ne_4+ie_1\}_{i\in\Z}$. A simple approximation of the characteristic $\frac{(\rho_-)^2}{(1-\rho_-)^2}$ shows that
	\begin{equation}
		c^0_1= (1-t_r)N -8rt_rN^{2/3}+\cO(N^{1/3}).
	\end{equation}
	It follows that
	\begin{equation}
\begin{aligned}
		c^0_1-v_1&\geq r_0t_r^{2/3} N^{2/3}-8rt_rN^{2/3}(1+o(1))\\
&=\Big(r_0\delta(\log(\delta^{-1}))^{-2}-2\delta^{3/2}(\log(\delta^{-1}))^{-2}(1+o(1))\Big)N^{2/3}.
\end{aligned}
	\end{equation}
	For $\delta$ small enough
	\begin{equation}
		c^0_1-v_1\geq \frac12 r_0\delta(\log(\delta^{-1}))^{-2}N^{2/3}
	\end{equation}
	and
	\begin{equation}
		\frac{c^0_1-v_1}{(t_r N)^{2/3}}\geq \frac{r_0}2.
	\end{equation}
	It follows from Lemma~\ref{lem:sb} that
	\begin{equation}\label{eq3}
		\P(Z^{\rho_-}_{v,Ne_4}<0)\leq e^{-cr_0^3}.
	\end{equation}
	Plugging \eqref{eq} and \eqref{eq3} in \eqref{eq4} and using \eqref{eq5} implies the result.
\end{proof}
Let $o_2=Ne_4$ and define
\begin{equation}
\hat{I}_i=\widehat{G}_{o_2,(1-t_r)N\mathrm{e}_4-(i+1)\mathrm{e}_1}-\widehat{G}_{o_2,(1-t_r)N\mathrm{e}_4-i\mathrm{e}_1} \quad\textrm{for }i\in \Z
\end{equation}
and let $(\arr_j)_{j\in\Z}$ and $(\serv_j)_{j\in\Z}$ be two independent sequences of i.i.d.\ random variables, independent of $\hat{I}$, such that
\begin{equation}
a_0\sim \text{Exp}(1-{\rho_+}),\quad
s_0\sim \text{Exp}(1-{\rho_-}).
\end{equation}
For $i\in\Z$ define the shifted random varaibles
\begin{equation}
X^1_i=\hat{I}_i-2,\quad X^2_i=s_i-2,\quad X^3_i=a_i-2.
\end{equation}
Finally define the random walks
\begin{equation}
\begin{aligned}
&S^{a,x}_i=\sum_{k=x}^{i} X^a_k, \quad \textrm{for }a\in\{1,2,3\},\\
&S^{a,b,x}_i=S^{a,x}_i-S^{b,x}_i \quad \textrm{for }a,b\in\{1,2,3\}.
\end{aligned}
\end{equation}
In particular,
\begin{equation}
S^{2,3,x}_i=\sum_{k=x}^{i}(s_k-a_k).
\end{equation}
We also define unbiased versions of $S^{a,x}$ for $a\in\{2,3\}$
\begin{equation}
\bar{S}^{2,x}_i=\sum_{k=x}^{i}(s_k-(1-{\rho_-})^{-1}),\quad
\bar{S}^{3,x}_i=\sum_{k=x}^{i}(a_k-(1-{\rho_+})^{-1}).
\end{equation}
A simple computation gives, for $i>x$,
\begin{align}\label{ba}
S^{2,x}_i&\leq \bar{S}^{2,x}_i\leq S^{2,x}_i +(i-x)\frac{ rN^{-1/3}}{\frac14+\frac12rN^{-1/3}}\\
\bar{S}^{3,x}_i&\leq S^{3,x}_i\leq \bar{S}^{3,x}_i +(i-x)\frac{ rN^{-1/3}}{\frac14-\frac12rN^{-1/3}}.\label{ba2}
\end{align}

Our next result controls the maximum of $S^{1,z_0}$ on $I$.
\begin{lemma}\label{lem:com}
	There exists $c>0$ such that for any fixed $y>9r_0^{1/2}$
	\begin{equation}
	\P\Big(\sup_{z_0\leq i\leq z_1} |S^{1,z_0}_i|\leq y(z_1-z_0)^{1/2}\Big) \geq 1-2Ce^{-c(y-9r_0^{1/2})^2}-2e^{-cr_0^3}.
	\end{equation}	
for all $N$ large enough.
\end{lemma}
\begin{proof}
	Let
	\begin{equation}
		\lambda_\pm=\frac12\pm r_0(t_rN)^{-1/3}.
	\end{equation}
	We first show that
\begin{equation}\label{ine2}
	\begin{aligned}
		\P(\widehat Z^{\lambda_+}_{N\mathrm{e}_4,x}>0,\widehat Z^{\lambda_-}_{N\mathrm{e}_4,x}<0 \quad\forall x\in I)\geq 1-2e^{-cr_0^3}.
	\end{aligned}
\end{equation}
We will show that
\begin{equation}\label{ine}
	\P(\widehat Z^{\lambda_+}_{N\mathrm{e}_4,x}>0\quad\forall x\in I)\geq 1-e^{-cr_0^3},
\end{equation}
a similar result can be shown for $\hat{Z}^{\lambda_-}_{N\mathrm{e}_4,x}$ and the result follows by union bound. Note that
\begin{equation}
	\P(\widehat Z^{\lambda_+}_{N\mathrm{e}_4,x}>0\quad\forall x\in I)=\P(\widehat Z^{\lambda_+}_{N\mathrm{e}_4,(1-t_r)Ne_4+z_1e_1}>0).
\end{equation}
Consider the characteristic $\lambda_+$ emanating from $Ne_4$ and its intersection point $d$ with the set $\{(1-t_r)Ne_4+ie_1\}_{i\in\Z}$. A simple approximation of the characteristic $\frac{(\lambda_+)^2}{(1-\lambda_+)^2}$ shows that
\begin{equation}
d_1= (1-t_r)N +8r_0(t_rN)^{2/3}+\cO(N^{1/3})\geq (1-t_r)N+2r_0 (t_r N)^{2/3}
\end{equation}
for all $N$ large enough.
It follows that
\begin{equation}
d_1-[(1-t_r)N+z_1]\geq r_0(t_r N)^{2/3}.
\end{equation}
As in previous proofs, applying Lemma~\ref{lem:sb} we conclude \eqref{ine} and therefore \eqref{ine2}.

Define
\begin{equation}
\begin{aligned}
	\hat{I}^{\lambda_-}_i&=G^{\lambda_-}_{Ne_4,(1-t_r)Ne_4+ie_1}-G^{\lambda_-}_{Ne_4,(1-t_r)Ne_4+(i+1)e_1},\\
	\hat{I}^{\lambda_+}_i&=G^{\lambda_+}_{Ne_4,(1-t_r)Ne_4+ie_1}-G^{\lambda_+}_{Ne_4,(1-t_r)Ne_4+(i+1)e_1}.
\end{aligned}
\end{equation}
Using the Comparison Lemma, see Section~\ref{SectCompLemma}, we obtain
\begin{equation}
	\P(\hat{I}^{\lambda_-}_i\leq \hat{I}_i \leq \hat{I}^{\lambda_+}_i \quad i\in I)\geq 1-2e^{-cr_0^3}.
\end{equation}
Therefore, we also have
\begin{equation}\label{ine3}
	\P\Big(\sum_{k=z_0}^{i}(\hat I^{\lambda_-}_i-2)\leq S^{1,z_0}_i \leq \sum_{k=z_0}^{i}(\hat I^{\lambda_+}_i-2)\Big)\geq 1-2e^{-cr_0^3}.
\end{equation}

Denote $\widehat S_i:=\sum_{k=z_0}^{i}(\hat I^{\lambda_+}_k-(1-\lambda_+)^{-1})$, which is a martingale starting at time $i=z_0$. Then
\begin{equation}
\begin{aligned}
	&\P\Big(\sup_{z_0\leq i\leq z_1}\sum_{k=z_0}^{i}(\hat I^{\lambda_+}_k-2)\leq yr_0^{1/2}(t_rN)^{1/3}\Big)\\
	&\geq\P\Big(\sup_{z_0\leq i\leq z_1}\widehat S_i+(z_1-z_0)((1-\lambda_+)^{-1}-2)\leq yr_0^{1/2}(t_rN)^{1/3}\Big)\\
	&\geq \P\Big(\sup_{z_0\leq i\leq z_1}r_0^{-1/2}(t_rN)^{-1/3}\widehat S_i\leq y-9r_0^{1/2}\Big),
\end{aligned}
\end{equation}
where in the third line we used
\begin{equation}
	(z_1-z_0)((1-\lambda_+)^{-1}-2)= 8r_0(t_rN)^{1/3}+\cO(1)\leq 9r_0(t_rN)^{1/3}
\end{equation}
for all $N$ large enough. The scaling is chosen such that, setting $i=z_0+2 \tau r_0(t_r N)^{2/3}$, the scaled random walk $r_0^{-1/2}(t_rN)^{-1/3}\widehat S_i$ converges as $N\to\infty$ weakly to a Brownian motion on the interval $\tau\in [0,1]$ for some (finite) diffusion constant. Thus there exists constants $C,c>0$ such that for any given $y>9 r_0^{1/2}$
\begin{equation}\label{ine4}
	\P\Big(\sup_{z_0\leq i\leq z_1}\sum_{k=z_0}^{i}(\hat I^{\lambda_+}_k-2)\leq yr_0^{1/2}(t_rN)^{1/3}\Big)\geq 1-Ce^{-c(y-9r_0^{1/2})^2}
\end{equation}
for all $N$ large enough. Similarly we show that for any given $y>9r_0^{1/2}$
\begin{equation}\label{ine5}
\P\Big(\inf_{z_0\leq i\leq z_1}\sum_{k=z_0}^{i}(\hat I^{\lambda_-}_k-2)\geq -yr_0^{1/2}(t_rN)^{1/3}\Big)\geq 1-Ce^{-c(y-9r_0^{1/2})^2}
\end{equation}
for all $N$ large enough. From \eqref{ine3}, \eqref{ine4} and \eqref{ine5} it follows that
\begin{equation}
\begin{aligned}
	&\P\Big(\sup_{z_0\leq i\leq z_1} |S^{1,z_0}_i|\leq y(z_1-z_0)^{1/2}\Big)=\P\Big(\sup_{z_0\leq i\leq z_1} |S^{1,z_0}_i|\leq yr_0^{1/2}(t_rN)^{1/3}\Big) \\
	&\geq\P\Big( \inf_{z_0\leq i\leq z_1}\sum_{k=z_0}^{i}(\hat I^{\lambda_-}_k-2)\geq -yr_0^{1/2}(t_rN)^{1/3},\sup_{z_0\leq i\leq z_1}\sum_{k=z_0}^{i}(\hat I^{\lambda_+}_k-2)\leq yr_0^{1/2}(t_rN)^{1/3}\Big)
	\\
	&\geq 1-2e^{-c(y-9r_0^{1/2})^2}-2e^{-cr_0^3}.
\end{aligned}
\end{equation}
\end{proof}

Next we control the fluctuations of the random walk $S^{2,z_0}$.
\begin{lemma}\label{lem:Llb1}
Let $\cC=\{\sup_{z_0\leq i \leq z_1}|S^{2,z_0}_i| \leq y(z_1-z_0)^{1/2} \}$. Under the choice of parameters in \eqref{eq4.31}, there exists $c,\delta_0>0$ such that for $\delta<\delta_0$, and any fixed $y>r_0^{1/2}$
	\begin{equation}\label{ine1}
	\P(\cC)>1-Ce^{-c(y-r_0^{1/2})^2}
	\end{equation}
for all $N$ large enough.
\end{lemma}
\begin{proof}
	By \eqref{ba} we have
	\begin{equation}
	\P(\cC )>\P\Big(\sup_{z_0\leq i \leq z_1}|(z_1-z_0)^{-1/2} \bar{S}^{2,z_0}_i| \leq y-(z_1-z_0)^{1/2}\frac{ rN^{-{1/3}}}{\frac14+\frac12rN^{-1/3}}\Big).
	\end{equation}
	Note that
\begin{equation}\label{eq6.52}
	\begin{aligned}
		(z_1-z_0)^{1/2}\frac{ rN^{-{1/3}}}{\frac14+\frac12rN^{-1/3}}&=(2r_0)^{1/2}(t_rN)^{1/3}\frac{ rN^{-{1/3}}}{\frac14+\frac12rN^{-1/3}}
=(2r_0)^{1/2}\frac{t_r^{1/3} r}{\frac14+\frac12rN^{-1/3}}\\
&=(2r_0)^{1/2}\frac{\tfrac14\delta^{1/2}}{\frac14+\frac12rN^{-1/3}}.
	\end{aligned}
\end{equation}
Thus for all $N$ large enough and $\delta$ small enough
	\begin{equation}
		\P(\cC )>\P\Big(\sup_{z_0\leq i \leq z_1}|(z_1-z_0)^{-1/2} \bar{S}^{2,z_0}_i| \leq y-r_0^{1/2}\Big).
	\end{equation}
Also notice that $(z_1-z_0)^{-1/2} \bar{S}^{2,z_0}_i$ converges weakly to a Brownian motion as $N\to\infty$. Using Doob maximum inequality one deduces that for $N$ large enough
	 \eqref{ine1} indeed holds.
\end{proof}

For $M>0$ define
\begin{equation}
	\cE_2=\Big\{\sup_{i\in I}|S^{2,z_0}_i| \leq M(z_1-z_0)^{1/2} \Big\}.
\end{equation}
For $x>0$, define the sets
\begin{equation}
\begin{aligned}
\cE_1&=\Big\{\cH^{{\rho_-}}\in I_-,\sup_{i\in I}|S^{2,1,z_0}_i|\leq 2M(z_1-z_0)^{1/2}\Big\},\\
\cE_{3,x}&=\Big\{\inf_{i\in I_-} (x+S^{2,3,z_0}_i) > -2M(z_1-z_0)^{1/2} ,\inf_{i\in I_+} (x+S^{2,3,z_0}_i)<-7M(z_1-z_0)^{1/2}\Big\}.
\end{aligned}
\end{equation}
Note that on the event $\cE_2$
\begin{equation}
S^{2,1,z_0}_i>-M(z_1-z_0)^{1/2}-S^{1,z_0}_i\textrm{ and }S^{2,1,z_0}_i<M(z_1-z_0)^{1/2}-S^{1,z_0}_i,
\end{equation}
which implies
\begin{equation}\label{subs}
\cE_1\cap \cE_2 \supseteq \{\cH^{{\rho_-}}\in I_-,\sup_{i\in I}|S^{1,z_0}_i|\leq M(z_1-z_0)^{1/2}\}\cap\cE_2.
\end{equation}
Similarly, on the event $\cE_2$
\begin{equation}\label{eq4.42}
S^{2,3,z_0}_i>-M(z_1-z_0)^{1/2}-S^{3,z_0}_i\textrm{ and }
S^{2,3,z_0}_i<M(z_1-z_0)^{1/2}-S^{3,z_0}_i,
\end{equation}
so that, for $x\in[0,(z_1-z_0)^{1/2}]$ and $M\geq 1$,
\begin{equation}
\begin{aligned}\label{subs2}
\cE_{3,x}\cap \cE_2
&\supseteq\{\inf_{i\in I_-} (x-S^{3,z_0}_i) > -M(z_1-z_0)^{1/2} ,\inf_{i\in I_+} (x-S^{3,z_0}_i)<-8M(z_1-z_0)^{1/2}\}\cap \cE_2\\
&\supseteq \{\sup_{i\in I_-} S^{3,z_0}_i < M(z_1-z_0)^{1/2} ,\sup_{i\in I_+} S^{3,z_0}_i>9M(z_1-z_0)^{1/2}\}\cap \cE_2\\
&\supseteq \{\sup_{i\in I_-} S^{3,z_0}_i < M(z_1-z_0)^{1/2} ,\sup_{i\in I_+} S^{3,z_0}_i-S^{3,z_0}_0>9M(z_1-z_0)^{1/2}-S^{3,z_0}_0\}\cap \cE_2\\
&\supseteq \{\sup_{i\in I_-} |S^{3,z_0}_i| < M(z_1-z_0)^{1/2} ,\sup_{i\in I_+} S^{3,0}_i>10M(z_1-z_0)^{1/2}\}\cap \cE_2\\
&\supseteq \{\sup_{i\in I_-} |S^{3,z_0}_i| < M(z_1-z_0)^{1/2} , S^{3,0}_{z_1}>10M(z_1-z_0)^{1/2}\}\cap \cE_2,
\end{aligned}
\end{equation}
as in the first line we used \eqref{eq4.42}, in the second $x=0$ for the first term and $x=(z_1-z_0)^{1/2}$ for the second one.

Next we are going to prove that, conditioned on $\cE_2$, $\cE_1$ and $\cE_{3,x}$ occurs with positive probability.

\begin{lemma}\label{lem:Llb2} There exists $c_2,r_0>0$ and $M\geq 1$ such that for $x\in [0,(z_1-z_0)^{1/2}]$
	\begin{equation}
	\P(\cE_1,\cE_{3,x}|\cE_2)>c_2.
	\end{equation}
for all $N$ large enough.
\end{lemma}
\begin{proof}
	Define
\begin{equation}
	\begin{aligned}
	&\cF_1=\{\cH^{{\rho_-}}\in I_-,\sup_{i\in I}|S^{1,z_0}_i|\leq M(z_1-z_0)^{1/2}\},\\
	&\cF_3=\{\sup_{i\in I_-} |S^{3,z_0}_i| < M(z_1-z_0)^{1/2} , S^{3,0}_{z_1}>10M(z_1-z_0)^{1/2}\}.
	\end{aligned}
\end{equation}
	By \eqref{subs} and \eqref{subs2} and the independence of $S^{1,z_0}$ and $S^{3,z_0}$
	\begin{equation}\label{atb}
	\P(\cE_1,\cE_{3,x}|\cE_2)\geq \P(\cF_1,\cF_3|\cE_2)=\P(\cF_1)\P(\cF_3).
	\end{equation}
Next we want to derive lower bounds for $\P(\cF_1)$ and $\P(\cF_3)$.

Note that
\begin{equation}
	\P(\cF_1)\geq \P(\cH^{{\rho_-}}\in I_-)-\P(\sup_{i\in I}|S^{1,z_0}_i|\geq M(z_1-z_0)^{1/2}).
\end{equation}
By Lemma~\ref{lem:gl} and Lemma~\ref{lem:com} there exists $r_0>0$ and $M>9r_0^{1/2}$ for which
\begin{equation}
	\P(\cH^{{\rho_-}}\in I_-)\geq 1/4\quad \textrm{ and }\quad\P(\sup_{i\in I}|S^{1,z_0}_i|\geq M(z_1-z_0)^{1/2})\leq 1/8,
\end{equation}
so that
\begin{equation}\label{f1}
	\P(\cF_1)\geq 1/8.
\end{equation}

	Let us now try to find a lower bound for $\P(\cF_3)$.
	\begin{align}
	&\P(\cF_3)=\P\Big(\sup_{i\in I_-} |S^{3,z_0}_i| < M(z_1-z_0)^{1/2} ,S^{3,0}_{z_1}>10M(z_1-z_0)^{1/2}\Big)\\
	&=\P\Big(\sup_{i\in I_-} |S^{3,z_0}_i| < M(z_1-z_0)^{1/2}\Big)\P\Big( S^{3,0}_{z_1}>10M(z_1-z_0)^{1/2}\Big)
	\end{align}
	since the processes $\{S^{3,z_0}_{i}\}_{i\in I_-}$ and $\{S^{3,0}_{i}\}_{i\in I_+}$ are independent.
Note that  the centered random walks $\{(z_1-z_0)^{-1/2}\bar S^{3,z_0}_{i}\}_{i\in I_-}$ and $\{(z_1-z_0)^{-1/2}\bar S^{3,0}_{i}\}_{i\in I_+}$ converge weakly to a Brownian motion.
Furthermore, the difference coming from the non-zero drift of $S^{3,z_0}_i$ is, by \eqref{ba2}, bounded by
\begin{equation}
\sup_{i\in I}|(z_1-z_0)^{-1/2}S^{3,z_0}_i-(z_1-z_0)^{-1/2}\bar S^{3,z_0}_i|\leq (z_1-z_0)^{1/2}\frac{r N^{-1/3}}{\frac14-\frac12 r N^{-1/3}}
\end{equation}
which is $o\Big((z_1-z_0)^{1/2}\Big)$ as $\delta\rightarrow0$(similarly to \eqref{eq6.52}).
Thus by choosing $M$ large enough, there exists  $c_2>0$ such that for $N$ large enough
\begin{equation}
\begin{aligned}
		&\P\Big(\sup_{i\in I_-} |S^{3,z_0}_i| < M(z_1-z_0)^{1/2}\Big)\geq 1/2,\\
		&\P\Big( S^{3,0}_{z_1}>10M(z_1-z_0)^{1/2}\Big)\geq 16c_2.
	\end{aligned}
\end{equation}
	I follows that
	\begin{align}\label{f3}
		\P(\cF_3)\geq 8c_2
	\end{align}
	Plugging \eqref{f1} and \eqref{f3} in \eqref{atb} we obtain the result.
\end{proof}

Now we can prove the main statement of this section.
\begin{proof}[Proof of Proposition~\ref{prop:Lb}]
	Note that
\begin{equation}
\Big\{\cH^{{\rho_-}}\in I_-,\sup_{i\in I_-}\sum_{k=z_0}^i(\hat I_k^{{\rho_+}}-\hat{I}_k)<\sup_{i\in I_+}\sum_{k=z_0}^i(\hat I_k^{{\rho_+}}-\hat{I}_k)\Big\}
\subseteq \Big\{	\cH^{{\rho_-}}\in I_-, \cH^{{\rho_+}}>0 \Big\}.
\end{equation}
Indeed, if $\cH^{{\rho_-}}\in I_-$ then also $\cH^{{\rho_+}}\geq z_0$, and the second condition implies that $\cH^{{\rho_+}}\not\in I_-$.
Using a decomposition as in \eqref{eqEjs}, we can write
\begin{equation}
	\begin{aligned}
	&\Big\{\cH^{{\rho_-}}\in I_-,\sup_{i\in I_-}\sum_{k=z_0}^i(\hat I_k^{{\rho_+}}-\hat{I}_k)<\sup_{i\in I_+}\sum_{k=z_0}^i(\hat I_k^{{\rho_+}}-\hat{I}_k)\Big\}\\
	=&\Big\{\cH^{{\rho_-}}\in I_-, \sup_{i\in I_-}\sum_{k=z_0}^ie_k+\sum_{k=z_0}^i(\hat I_k^{{\rho_-}}-\hat{I}_k)<\sup_{i\in I_+}\sum_{k=z_0}^ie_k+\sum_{k=z_0}^i(\hat I_k^{{\rho_-}}-\hat{I}_k)\Big\}\label{i}
	\end{aligned}
\end{equation}
	where $e_j=\hat I^{\rho_+}_j-\hat I^{\rho_-}_j$. By \eqref{sume} the following event has the same probability as \eqref{i}
	\begin{multline}
	\cE_4=\Big\{\cH^{{\rho_-}}\in I_-,\\
\sup_{i\in I_-}\Big(\inf_{z_0\leq l\leq i}w_{z_0-1}+S^{2,3,z_0}_l\Big)^-+S^{2,1,z_0}_i<\sup_{i\in I_+}\Big(\inf_{z_0\leq l\leq i}w_{z_0-1}+S^{2,3,z_0}_l\Big)^-+S^{2,1,z_0}_i\Big\}.
	\end{multline}
	It follows that
	\begin{equation}\label{Llb5}
	\P(\cH^{{\rho_-}}\in I_-, \cH^{{\rho_+}}>0) \geq \P(\cE_4).
	\end{equation}

	For $x>0$, define
	\begin{multline}
	\cE_{4,x}=\Big\{\sup_{i\in I_-}S^{2,1,z_0}_i>\sup_{i\in I_+}S^{2,1,z_0}_i,\\
\sup_{i\in I_-}\Big(\inf_{z_0\leq l\leq i}x+S^{2,3,z_0}_l\Big)^-+S^{2,1,z_0}_i<\sup_{i\in I_+}\Big(\inf_{z_0\leq l\leq i}x+S^{2,3,z_0}_l\Big)^-+S^{2,1,z_0}_i\Big\}.
	\end{multline}
	Note that
	\begin{equation}
	\cE_2\cap\cE_1\cap\cE_{3,x} \subseteq \cE_{4,x}.
	\end{equation}

In our case, the value of $x$ in $\cE_{4,x}$ is random and distributed according to \eqref{des}. Therefore we have
\begin{equation}
	\begin{aligned}
	\P(\cE_4)&=\int_0^\infty\P(\cE_4|w_{z_0-1}=w)f(dw)=\int_0^\infty\P(\cE_{4,w})f(dw)\\
	&\geq \int_0^\infty\P(\cE_2,\cE_1,\cE_{3,w})f(dw)\geq \int_0^{(z_1-z_0)^{1/2}}\P(\cE_2,\cE_1,\cE_{3,w})f(dw),\label{Llb}
	\end{aligned}
\end{equation}
	where in the second equality we used the fact that the processes $\{S^{i,z_0}_j\}_{j\geq z_0,i\in\{1,2,3\}}$ are independent of $w_{z_0-1}$.

Taking $M$ large enough, Lemma~\ref{lem:Llb1} gives
	\begin{equation}\label{e2b}
		\P(\cE_2)\geq 1/2
	\end{equation}
for all $N$ large enough. \eqref{e2b} and Lemma~\ref{lem:Llb2} imply that there exists $c_4>0$ such that for any $w\in [0,(z_1-z_0)^{1/2}]$, $\delta<\delta_0$, and large enough $N$
	\begin{equation}\label{Llb4}
	\P(\cE_2,\cE_1,\cE_{3,w})\geq \tfrac12 c_2.
	\end{equation}
	Plugging \eqref{Llb4} in \eqref{Llb}
	\begin{equation}
	\P(\cE_4)\geq \tfrac12 c_2\Big(1-\frac{\frac12-rN^{-1/3}}{\frac12+rN^{-1/3}}e^{-2rN^{-1/3}(z_1-z_0)^{1/2}}\Big).
	\end{equation}
	Note that
	\begin{equation}
		2rN^{-1/3}(z_1-z_0)^{1/2}=2r (2r_0)^{1/2}t_r^{1/3}=2^{3/2} r_0^{1/2}\delta^{1/2}\rightarrow 0
	\end{equation}
	if $r_0 \delta\rightarrow 0$ as $\delta\to 0$. Then by first order approximation of the exponential function, there exists $C>0$ such that for $\delta<\delta_0$, $r_0\leq \delta^{-1}(\log\delta^{-1})^{-1}$, and large enough $N$
	\begin{equation}\label{Llb6}
	\P(\cE_4)\geq C\delta^{1/2}.
	\end{equation}
	Using \eqref{Llb6} in \eqref{Llb5} we obtain the result.
\end{proof}

\subsection{Proof of Theorem~\ref{thm:LBcoal}}
Finally we prove the second main result of this paper.
\begin{proof}[Proof of Theorem~\ref{thm:LBcoal}]
	Let
	\begin{equation}
		\rho'_+=\frac12+\frac1{120}rN^{-1/3},\quad
		\rho'_-=\frac12-\frac1{120}rN^{-1/3},
	\end{equation}
	and define
	\begin{equation}
		\cA'=\left\{-\frac r8 N^{2/3} \leq Z^{ \rho'_-}_{o,N\mathrm{e}_4}\leq Z^{\rho'_+}_{o,N\mathrm{e}_4}\leq \frac r8 N^{2/3}\right\}.
	\end{equation}
	By Lemma~\ref{lem:sb}
	\begin{equation}\label{Ap}
		\P(\cA')\geq 1-e^{-cr^3}
	\end{equation}
	Define
	\begin{equation}
		y^1=\frac14rN^{2/3} \mathrm{e}_1,\quad
		y^2=\frac14rN^{2/3} \mathrm{e}_2.
	\end{equation}
	Let $o^2=N\mathrm{e}_4$. Similar to \eqref{ex} we define
	\begin{equation}
	H^j=\sup\{i:(i,(1-t_r)N)\in \pi_{y^j,o^2}\} \quad \textrm{for }i\in\{1,2\}.
	\end{equation}
	Note that on the event $\cA'$, the geodesics $\pi^{\rho'_-}_{o,o^2},\pi^{\rho'_+}_{o,o^2}$ are sandwiched between the geodesics $\pi_{y^1,o^2},\pi_{y^2,o^2}$, which implies that if the geodesics $\pi^{\rho'_-}_{o,o^2},\pi^{\rho'_+}_{o,o^2}$ did not coalesce then neither did $\pi_{y^1,o^2},\pi_{y^2,o^2}$ i.e.\
	\begin{equation}\label{cont}
		\cA'\cap \{\cH^{\rho'_-}\in I_-, \cH^{\rho'_+}>0\} \subseteq \{C_p(\pi_{y^1,o^2},\pi_{y^2,o^2})>L_{1-t_r}\}.
	\end{equation}
	Indeed, on the event $\cA'$
	\begin{equation}
		-y^2_2<Z^{\rho_-}_{o,o^2}\leq Z^{\rho_+}_{o,o^2} \leq y^1_1
	\end{equation}
	so that
	\begin{equation}
		\pi_{y_2,o^2} \preceq \pi^{\rho'_-}_{o,o^2}\preceq \pi^{\rho'_+}_{o,o^2} \preceq \pi_{y_1,o^2}
	\end{equation}
	 which implies that under $\cA'\cap \{\cH^{\rho'_-}\in I_-, \cH^{\rho'_+}>0\}$
	 \begin{equation}
	 	L_{1-t_r}\leq C_p(\pi^{\rho'_+}_{o,o^2},\pi^{\rho'_-}_{o,o^2})\leq C_p(\pi_{y^1,o^2},\pi_{y^2,o^2})
	 \end{equation}
	 Note that as $y^1,y^2\in \cR^{r/2,1/4}$ and $o^2\in \cC^{s_r/2,t_r}$
	\begin{equation}\label{eq8}
		\{C_p(\pi_{y^1,o^2},\pi_{y^2,o^2})>L_{1-t_r}\}\subseteq \{\exists x \in \cC^{s_r/2,t_r},y\in\cR^{r/2,1/4}: C_p(\pi^{1/2}_{o,x},\pi_{y,x})>L_{1-t_r}\}.
	\end{equation}
	Indeed, if the geodesics $\pi_{y^1,o^2}$ and $\pi_{y^2,o^2}$ do not meet before the time horizon $L_{1-t_r}$, at least one of the them did not coalesce with the geodesic $\pi^{1/2}_{o,o^2}$ before $L_{1-t_r}$. It follows from \eqref{cont}, \eqref{eq8} and \eqref{Ap} that
	 \begin{equation}
	 	\P(\exists x \in \cC^{s_r/2,t_r},y\in\cR^{r/2,1/4}: C_p(\pi^{1/2}_{o,x},\pi_{y,x})>L_{1-t_r})\geq C\delta^{1/2}.
	 \end{equation}
\end{proof}
\subsection{Proof of Theorem~\ref{thm:coal2}}
\begin{proof}[Proof of Theorem~\ref{thm:coal2}]
	Note that
	\begin{equation}\label{inc}
\begin{aligned}
		&\{C_p(\pi^{1/2}_{o,x},\pi_{y,x})\leq L_{1-\tau} \quad \forall x\in \cC^{\delta,\tau},y\in \cR^{\frac18\log\delta^{-1},1/4}\}\\
		&\subseteq \{C_p(\pi_{w,x},\pi_{y,x})\leq L_{1-\tau} \quad \forall x\in \cC^{\delta,\tau},w,y\in \cR^{\frac18\log\delta^{-1},1/4}\}.
\end{aligned}
\end{equation}
	Indeed, on the event that any geodesic starting from $\cR^{\frac18\log\delta^{-1},1/4}$ and terminating in $\cC^{\delta,\tau}$ coalesces with the stationary geodesic before the time horizon $L_{1-\tau}$ any two geodesics starting from $\cR^{\frac18\log\delta^{-1},1/4}$ and terminating in $\cC^{\delta,\tau}$ must coalesce as well. Theorem~\ref{thm:coal} and \eqref{inc} imply the lower bound in Theorem~\ref{thm:coal2}.

Next note that
\begin{equation}
	\begin{aligned}\label{inc2}
	&\{\exists x\in \cC^{\delta,\tau},y\in \cR^{\frac18\log\delta^{-1},1/4}: C_p(\pi^{1/2}_{o,x},\pi_{y,x})> L_{1-\tau}, |Z^{1/2}_{o,x}|\leq \frac1{16}\log(\delta^{-1})N^{2/3} \}\\
	&\subseteq \{C_p(\pi_{w,x},\pi_{y,x})\leq L_{1-\tau} \quad \forall x\in \cC^{\delta,\tau},w,y\in \cR^{\frac18\log\delta^{-1},1/4}\}^c.
	\end{aligned}
\end{equation}
	To illustrate the validity of \eqref{inc2}, assume w.l.o.g.\ that $x=Ne_4,y=o$ such that $C_p(\pi^{1/2}_{o,e_4N},\pi_{o,e_4N})>L_{1-\tau}$ and that $\frac1{16}\log(\delta^{-1})N^{2/3}\geq Z^{1/2}_{o,Ne_4}=a>0$. It follows that
	\begin{equation}\label{set}
		C_p(\pi_{ae_1,Ne_4},\pi_{o,Ne_4})> L_{1-\tau}
	\end{equation}
	holds. The event in \eqref{set} is contained in the event in the last line of \eqref{inc2} which implies \eqref{inc2}.

\eqref{inc2} implies that
\begin{equation}
	\begin{aligned}\label{ine7}
		&\P\Big(C_p(\pi^{1/2}_{o,x},\pi_{y,x})\leq L_{1-\tau} \quad \forall x\in \cC^{\delta,\tau},y\in \cR^{\frac18\log\delta^{-1},1/4}\Big)\\
		&+\P\Big(|Z^{1/2}_{o,x}|>\frac1{16}\log(\delta^{-1})N^{2/3} \quad \text{ for some }x\in \cC^{\delta,\tau}\Big)\\
		&\geq\P\Big(C_p(\pi_{w,x},\pi_{y,x})\leq L_{1-\tau} \quad \forall x\in \cC^{\delta,\tau},w,y\in \cR^{\frac18\log\delta^{-1},1/4}\Big).
	\end{aligned}
\end{equation}
	Next we claim that for some $c>0$
	\begin{equation}\label{ine6}
		\P\Big(|Z^{1/2}_{o,x}|>\frac1{16}\log(\delta^{-1})N^{2/3} \quad \text{ for some }x\in \cC^{\delta,\tau}\Big)\leq e^{-c\log(\delta^{-1})^3}.
	\end{equation}
	Indeed, it follows by \eqref{r1} with $r=\frac1{15}\frac1{16}\log(\delta^{-1})$ that
	\begin{equation}
 \begin{aligned}
		&\P\Big(Z^{1/2}_{o,x}>\frac1{16}\log(\delta^{-1})N^{2/3} \quad \text{ for some }x\in \cC^{\delta,\tau}\Big)\\
		&\leq \P\Big(Z^{\rho_+}_{o,x}>\frac1{16}\log(\delta^{-1})N^{2/3} \quad \text{ for some }x\in \cC^{\delta,\tau}\Big) \leq e^{-c\log(\delta^{-1})^3}.
	\end{aligned}
	\end{equation}
A similar bound can be obtained for the lower tail to obtain \eqref{ine6}.
	Using the upper bound in \eqref{thm:LBcoal} and \eqref{ine6} in \eqref{ine7} we obtain the upper bound in Theorem~\ref{thm:coal2}.
\end{proof}
\appendix

\section{An estimate}
\begin{lemma}\label{lem:crw}
	Let $0<\beta<\alpha<1$. Let
	\begin{equation}
	S_n=\sum_{i=1}^{n}X_i
	\end{equation}
	where $\{X_i\}_{i\geq1}$ are i.i.d.\ with law
	\begin{equation}
	X_i\sim\mathrm{Exp}(\alpha)-\mathrm{Exp}(\beta).
	\end{equation}
	Then
	\begin{equation}
	\P\Big(\sup_{i\geq 1}S_i>\lambda\Big)\leq \frac{\beta}{\alpha}e^{-(\alpha-\beta)\lambda}\quad \text{for all $\lambda>0$}.
	\end{equation}
\end{lemma}


\end{document}